\documentclass[12pt]{article}
\title{Jeu de taquin and a monodromy problem for Wronskians of polynomials}
\author{Kevin Purbhoo%
\footnote{Department of Combinatorics and Optimization, University of Waterloo,
Waterloo, Ontario, Canada; {\tt kpurbhoo@math.uwaterloo.ca}.}
}

\usepackage[mathscr]{eucal}
\usepackage{latexsym, amssymb, amsthm, amsmath}
\usepackage{epsfig, xspace}
\usepackage[usenames, dvips]{color}
\usepackage{ifthen}

\newboolean{blackandwhiteversion}
\setboolean{blackandwhiteversion}{false}                                 

\ifthenelse{\boolean{blackandwhiteversion}}{\renewcommand{\color}[1]{}}{}

\newcommand{\CC}{\mathbb{C}}
\newcommand{\FF}{\mathbb{F}}
\newcommand{\ZZ}{\mathbb{Z}}
\newcommand{\QQ}{\mathbb{Q}}
\newcommand{\RR}{\mathbb{R}}
\newcommand{\CP}{\mathbb{CP}}
\newcommand{\FP}{\mathbb{FP}}
\newcommand{\RP}{\mathbb{RP}}
\newcommand{\PP}{\mathbb{P}}

\newcommand{\puiseux}[1]{\CC\{\!\{#1\}\!\}}
\newcommand{\psK}{\mathcal{K}}
\newcommand{\psX}{\mathcal{X}}

\newcommand{\poln}{\CC_{n-1}[z]}
\newcommand{\Rpoln}{\RR_{n-1}[z]}
\newcommand{\Kpoln}{\psK_{n-1}[z]}

\newcommand{\pol}[1]{\CC_{#1}[z]}
\newcommand{\Rpol}[1]{\RR_{#1}[z]}
\newcommand{\Kpol}[1]{\psK_{#1}[z]}
\newcommand{\Fpol}[1]{\FF_{#1}[z]}

\newcommand{\SL}{\mathrm{SL}}
\newcommand{\Wr}{\mathrm{Wr}}
\newcommand{\Gr}{\mathrm{Gr}}
\newcommand{\Rect}{{\mathchoice%
{\raisebox{.05ex}{$\sqsubset\!\!\sqsupset$}}
{\raisebox{.05ex}{$\sqsubset\!\!\sqsupset$}}
{\sqsubset\!\!\sqsupset}
{\sqsubset\!\!\sqsupset}
}}
\newcommand{\lowBox}{{\mathchoice%
{{\raisebox{-0.2ex}{$\Box$}}}
{{\raisebox{-0.2ex}{$\Box$}}}
{\Box}
{\Box}
}}
\newcommand{\ordSYT}{\mathsf{SYT}}
\newcommand{\SYT}{\mathrm{SYT}}
\newcommand{\DIT}{\mathrm{DIT}}
\newcommand{\bolda}{{\bf a}}
\newcommand{\boldb}{{\bf b}}
\newcommand{\boldc}{{\bf c}}
\newcommand{\boldp}{{\bf p}}
\newcommand{\boldw}{{\bf w}}

\newcommand{\ordT}{\mathsf{T}}
\newcommand{\derz}{{\textstyle \frac{d}{dz}}}
\newcommand{\smallidmatrix}
{\left(\begin{smallmatrix} 1 & 0 \\ 0 & 1\end{smallmatrix}\right)}

\newcommand{\Proj}{\mathop{\mathrm{Proj}}}
\newcommand{\Spec}{\mathop{\mathrm{Spec}}}

\newcommand{\sgn}{\mathrm{sgn}}
\newcommand{\ord}{\mathrm{ord}}

\newcommand{\slide}{\mathrm{slide}}
\newcommand{\rectify}{\mathrm{rect}}

\newcommand{\val}{\mathrm{val}}
\newcommand{\leadterm}{\text{\rm \footnotesize LT}}
\newcommand{\leadcoeff}{\text{\rm \footnotesize LC}}
\newcommand{\initial}{\mathrm{In}}
\newcommand{\weight}{\mathrm{wt}}
\newcommand{\I}{{\rm (I)}\xspace}
\newcommand{\II}{{\rm (II)}\xspace}
\newcommand{\III}{{\rm (III)}\xspace}

\newcommand{\spc}[1]{{\mspace{2mu}#1\mspace{2mu}}}

\newenvironment{alignqed}
  {\begin{displaymath}\begin{aligned}[b]}
  {\end{aligned}\qedhere\end{displaymath}}

\newtheorem{lemma}{Lemma}[section]
\newtheorem{theorem}[lemma]{Theorem}
\newtheorem{example}[lemma]{Example}
\newtheorem{definition}[lemma]{Definition}
\newtheorem{remark}[lemma]{Remark}

\newtheorem{proposition}[lemma]{Proposition}
\newtheorem{corollary}[lemma]{Corollary}

\numberwithin{equation}{section}
\numberwithin{figure}{section}

\definecolor{DarkBlue}{rgb}{0, 0.1, 0.55}
\definecolor{DarkRed}{rgb}{0.45, 0, 0}
\newcommand{\bfdef}[1]{{\bf \color{DarkBlue} \emph{#1}}}
\newcommand{\pcolor}[1]{{\color{DarkBlue}#1}}
\newcommand{\ncolor}[1]{{\color{DarkRed}#1}} 
\newcommand{\te}[1]%
{
 \begin{picture}(25,25)(3,0)
 \put(0,0){\includegraphics{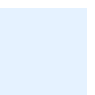}}
 \put(0,9){\makebox[25pt]{\pcolor{#1}}}
 \end{picture}
}
\newcommand{\TE}[1]%
{
 \begin{picture}(25,25)(3,0)
 \put(0,0){\includegraphics{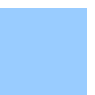}}
 \put(0,9){\makebox[25pt]{\pcolor{\bf{#1}}}}
 \end{picture}
}
\newcommand{\ue}[1]%
{
 \begin{picture}(25,25)(3,0)
 \ifthenelse{\boolean{blackandwhiteversion}}{}
 {\put(0,0){\includegraphics{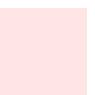}}}
 \put(0,9){\makebox[25pt]{\ncolor{#1}}}
 \end{picture}
}
\newcommand{\be}[1]%
{
 \begin{picture}(25,25)(3,0)
 \put(0,0){\includegraphics{lightblueshadebox.eps}}
 \put(0,9){\makebox[25pt]{#1}}
 \end{picture}
}

\begin{document}
\maketitle

\begin{abstract}
The Wronskian associates to $d$ linearly independent
polynomials of degree at most $n$, a non-zero 
polynomial of degree at most $d(n{-}d)$.  This can be viewed as
giving a flat, finite morphism from the Grassmannian $\Gr(d,n)$ 
to projective space of the same dimension.  
In this paper, we study the monodromy groupoid of this
map.  When the roots of the Wronskian are real, we show that
the monodromy is combinatorially 
encoded by Sch\"utzenberger's jeu de taquin;  hence we obtain new 
geometric
interpretations and proofs of a number of results from jeu de taquin 
theory, including the Littlewood-Richardson rule.
\end{abstract}


\section{Introduction}
\subsection{The Wronski map}

For any non-negative integer $m$,
let $\Fpol{m}$ denote the $(m{+}1)$-dimensional 
vector space of polynomials of degree at most $m$ over a field $\FF$:
$$\Fpol{m} := \{f(z) \in \FF[z] \mid \deg f(z) \leq m\}\,.  $$
Throughout, we fix integers $0<d<n$.  Let $X := \Gr_d(\poln)$ be
the Grassmannian whose points represent $d$-dimensional linear subspaces 
of $\poln$.  Let $N := d(n{-}d) = \dim X$ be its dimension.

Given polynomials
$f_1(z), \dots, f_d(z) \in \poln$, the Wronskian
$$\Wr_{f_1, \dots, f_d}(z) :=
\begin{vmatrix}
f_1(z) & \cdots & f_d(z)\\
f_1'(z) & \cdots  & f_d'(z) \\
\vdots &  \vdots & \vdots \\
f_1^{(d-1)}(z) & \cdots & f_d^{(d-1)}(z)
\end{vmatrix} $$
is a polynomial of degree at most $N$.
If $f_1, \dots, f_d$ are 
linearly dependent, the Wronskian is zero; otherwise up to a constant 
multiple, $\Wr_{f_1, \dots, f_d}(z)$ depends only on the linear span 
$\langle f_1(z), \dots, f_d(z) \rangle \subset \poln$.  Thus the
Wronskian gives a well defined morphism of schemes 
$\Wr: X \to \PP(\pol{N})$, called the \bfdef{Wronski map}.  For 
$x \in X$ we write $\Wr(x;z)$ for any representative of $\Wr(x)$ in
$\pol{N}$.

This morphism turns out to be extremely well behaved.
It appears in algebraic geometry in a number of different guises.
In the context of enumerating rational curves with prescribed 
ramifications,
Eisenbud and Harris proved the following theorem \cite{EH}:
\begin{theorem}
\label{thm:flatfinite}
$\Wr: X \to \PP(\pol{N})$ is a flat, finite morphism of schemes.
\end{theorem}

A point $x \in X$ is \bfdef{real} if the subspace of
$\poln$ represented by $x$ has a basis 
$f_1(z), \dots, f_d(z) \in \Rpoln$.
In 1995,
B. Shapiro and M. Shapiro made a remarkable conjecture concerning 
the reality of the fibres of $\Wr(x;z)$, which has been a source of
inspiration for much of the work relating to the Wronski map.
The conjecture (as refined by Sottile \cite{Sot2}) has
two parts, the first of which is given below and was proved in two
papers
by Mukhin, Tarasov and Varchenko \cite{MTV1, MTV2} (see also~\cite{GHY}).

\begin{theorem}
\label{thm:ssconj}
Let $g(z) \in \Rpol{N}$ be a polynomial with $N$ distinct real roots.
Then the fibre $\Wr^{-1}(g(z))$ is reduced and every point in 
the fibre is real.
\end{theorem}

Although the reality of the fibres is prominent in their proof,
the more pertinent fact for us is that these fibres are reduced;
the reality statement is a relatively simple consequence of 
this~\cite{Sot1}.
The second part
of the Shapiro-Shapiro conjecture concerns the multiplicities of the 
fibre when the roots of $g(z)$ are real but not distinct 
(see Remark~\ref{rmk:ssconjparttwo}).

\bigskip 

\noindent
In this paper, we study the monodromy groupoid of the Wronski map over the 
base of points where the fibre is reduced.  Specifically we will be
looking at a subgroupoid, which describes the lifting of certain
interesting paths and loops.
Our main goal is to show that these liftings are fundamentally related to 
Sch\"utzenberger's jeu de taquin~\cite{Sch}.  Through this relationship, 
we will see that much of the combinatorial
structure in jeu de taquin theory can be attributed to the 
geometric structure of the Wronski map.

\subsection{Outline of paper}

It is a classical result, originating with work of Castelnuovo~\cite{Cas}, 
that the fibres of the Wronski map can be 
interpreted as intersections of Schubert varieties.
We review this and other relevant background material
in Section~\ref{sec:background}.
From this interpretation, one can see that the degree
of the map $\Wr$ 
is given by counting standard Young tableaux whose shape is a
$d \times (n{-}d)$ rectangle, a calculation which dates 
back to Schubert~\cite{Schubert}.
We denote the set of all such tableaux by
$\ordSYT(\Rect)$.

Eremenko and Gabrielov~\cite{EG} showed that for
suitable base points in $\PP(\pol{N})$, there is in fact a natural 
way to index the points in the fibre of
$\Wr$ by $\ordSYT(\Rect)$.  Using Theorem~\ref{thm:ssconj}, the
notion of a suitable base point can be extended to any polynomial
with $N$ distinct real roots.   We will give a generalized and 
more explicit reformulation of this correspondence,
which will allow us to describe the monodromy for certain loops and 
paths in $\PP(\pol{N})$ in terms of tableaux.
To facilitate such a description, it will be helpful to modify our 
notion of standard Young tableau slightly, to allow entries in a 
field $\FF$ with a norm. 
As explained in Section~\ref{sec:jdt}, these enhancements allow us 
to speak of paths of tableaux, which, when $\FF = \RR$, can
be viewed as a mild extension of jeu de taquin.

In Section~\ref{sec:labelling}, we state and establish our formulation of
the correspondence.  Briefly, this works as follows: the Pl\"ucker 
coordinates of a point $x \in X$ are described in terms a tableau 
whose entries are the roots of $\Wr(x; -z)$.  If we work over the field 
of Puiseux
series $\puiseux{u}$, the tableau tells us the leading 
terms of the Pl\"ucker coordinates; over the complex numbers, 
this becomes an approximation.
Our approach is related to the types of arguments found 
in~\cite{EG, Sot1}, 
in that it can be interpreted as an asymptotic analysis 
over the real or complex numbers.

Using this correspondence, we can identify certain paths of
tableaux with paths in $X$. 
The most important example of this directly relates 
the monodromy problem 
to jeu de taquin theory.  We will show that for paths in 
$\PP(\pol{N})$ of polynomials whose roots are all real, 
the monodromy of $\Wr$ is described (in the
sense outlined in Section~\ref{sec:jdt}) by a sequence of
Sch\"utzenberger slides.  This result is formulated in
Section~\ref{sec:definesliding}, and proved in 
Section~\ref{sec:monodromy}.

A secondary example, also discussed in Section~\ref{sec:monodromy},
is the following.
For any positive integers $k,L$ such that
$1 \leq k < N$, and $L \geq 2$,
we can
define a permutation $s_{k,L}: \ordSYT(\Rect) \to \ordSYT(\Rect)$, 
as follows.  
For $\ordT \in \ordSYT(\Rect)$, $s_{k,L}(\ordT)$ is the tableau
obtained by swapping entries $k$ and $k{+}1$ in $\ordT$, if the total of
the horizontal and vertical distance between $k$ and $k{+}1$ equals
$L$; otherwise $s_{k,L}(\ordT) = \ordT$.
We will show that there exist loops in $\PP(\pol{N})$, such that
the monodromy of $\Wr$ is given by $s_{k,L}$.

These two results allow us, in Section~\ref{sec:LRrule},
to give geometric interpretations and proofs of a number 
of combinatorial theorems
involving jeu de taquin.
Among these is the Littlewood-Richardson rule.  
Our geometric interpretation of the Littlewood-Richardson rule
is notably different from those of Vakil~\cite{Vak} and 
Coskun~\cite{Cos}:
whereas their approaches involve degenerations of an 
intersection of two Schubert varieties, we begin by considering 
a general fibre of the Wronski map, which can be regarded as an 
intersection of $N$ Schubert varieties, and degenerating to
a special fibre, supported on a union 
of intersections of Schubert varieties
(cf. \eqref{eqn:unionofschuberts}).  We deduce the 
Littlewood-Richardson rule by showing that the combinatorics keeps track of 
multiplicities in each individual intersection of Schubert 
varieties comprising this union.

\subsection{Acknowledgements}
We thank Frank Sottile, Ravi Vakil, Alexander Varchenko, and 
Soroosh Yazdani for helpful
discussions, and Ian Goulden for comments on the manuscript.  This 
research was partially 
supported by an NSERC discovery grant.

\section{Background on the Wronski map}
\label{sec:background}

\subsection{Roots of the Wronskian and $\SL_2(\CC)$-action}

If $\bolda$ is a multiset and $S$ is a set, we say $\bolda$ is
a \bfdef{multisubset} of $S$ and write $\bolda \Subset S$ if every element 
of $\bolda$ is an element of $S$.  We write $\bolda \subset S$ if
every element of $\bolda$ has multiplicity $1$, i.e. $\bolda$ is a set.

As is suggested by Theorem~\ref{thm:ssconj},
it will be convenient to regard $\Wr(x;z)$ in terms of the multiset
of its roots.
If the degree of $\Wr(x;z)$ is strictly less than $N$, we will think 
of $\Wr(x;z)$ as having $N - \deg \Wr(x;z)$ roots at infinity.
If $\Wr(x;z) = \prod_{i=1}^k (z+a_i)$, let 
$\pi(x) := \{a_1, \dots, a_N\} \Subset \CP^1$, viewed as a multiset,
where $a_{k+1} = \dots = a_N = \infty$ if $k<N$.  Thus $\pi(x)$
is the multiset of roots of $\Wr(x; -z)$.

The group $\SL_2(\CC)$ acts on 
everything.  If 
$\phi = 
\left(\begin{smallmatrix}
\phi_{11} & \phi_{12} \\ \phi_{21} & \phi_{22}
\end{smallmatrix}\right) \in \SL_2(\CC)$,
we have the usual action on $\CP^1$,
$$\phi(w) := \frac{\phi_{11} w + \phi_{12}}{\phi_{21} w + \phi_{22}}$$
for $w \in \CP^1$, and hence an action on multisubsets of $\CP^1$.
On $\pol{m}$, we define the action as follows:
$$\phi f(z) := (\phi_{21} z + \phi_{11})^m 
f\big(\frac{\phi_{22} z + \phi_{12}}{\phi_{21} z + \phi_{11}}\big)$$
for $f(z) \in \pol{m}$.
The action on $\poln$ induces an action on $X$.  With these
definitions, the following 
proposition is straightforward to check.

\begin{proposition}
\label{prop:sl2equivariant}
For $\phi \in \SL_2(\CC)$ and $x \in X$ we have,
$\phi(\pi(x)) = \pi(\phi(x))$.
\end{proposition}

We will use the following notation to describe the fibres of
the Wronski map.
For a multiset $\bolda = \{a_1, \ldots, a_N\} \Subset \CP^1$, let
$X(\bolda) := \pi^{-1}(\bolda) = \{x \in X \mid \pi(x) = \bolda\}$.
Thus $X(\bolda)$ is the fibre of the map $\Wr$ at the point
$\prod_{a_i \neq \infty} (z+a_i)$.  
If $\bolda_t$, $t \in [0,1]$, is a path in the space of $N$-element 
multisubsets of $\CP^1$ such that the fibre $X(\bolda_t)$ is reduced
for all $t \in [0,1]$, we write $x_t \in X(\bolda_t)$ to describe
a lifting of this path to $X$.  If $x_0$ is specified, this
lifting is unique.  In particular, we associate to each 
$x_0 \in X(\bolda_0)$ a point $x_1 \in X(\bolda_1)$.  The \bfdef{monodromy} 
of the path $\bolda_t$ is the bijection $X(\bolda_0) \to X(\bolda_1)$ 
defined by this process.

When the roots of the Wronskian are real, we will generally restrict 
the action of $\SL_2(\CC)$ to the subgroup $\SL_2(\RR)$,
as exemplified in the following important corollary
of Theorem~\ref{thm:ssconj}.

%
%

\begin{corollary}
\label{cor:nomonodromy}
Let $\bolda_t$, $t \in [0,1]$ be a loop in the space of $N$-element
subsets of $\RP^1$.  Suppose there exists some $w \in \RP^1$ such
that $w \notin \bolda_t$ for all $t \in [0,1]$.  Then the monodromy
of $\bolda_t$ is trivial, i.e. the identity map.
\end{corollary}

\begin{proof}
First suppose $w = \infty$.
Let $Z \subset \PP(\pol{N})$ be the topological subspace of 
polynomials with exactly $N$ distinct real roots.
Then $\bolda_t$ encodes path in $Z$, which is a simply connected
space.  Since the fibres of the map $\Wr: \Wr^{-1}(Z) \to Z$
are reduced by Theorem~\ref{thm:ssconj}, the monodromy is necessarily 
trivial.

For other $w$, there exists 
$\phi \in \SL_2(\RR)$ such that $\phi(w) = \infty$.  From the first case,
we know that the monodromy of the loop $\phi(\bolda_t)$ is trivial, 
and the result follows from Proposition~\ref{prop:sl2equivariant}.
\end{proof}

\subsection{Partitions and Pl\"ucker coordinates on $X$}

Let $\Lambda$ denote the set of all partitions whose
whose diagrams fit inside a $d \times (n{-}d)$ rectangle.  Formally,
these are decreasing sequences of integers
$\lambda = (\lambda_1 \geq \dots \geq \lambda_d)$, where
$n{-}d \geq \lambda_1$ and $\lambda_d \geq 0$. 
We will draw the diagram of $\lambda \in \Lambda$ in the English 
convention, with $\lambda_1$ boxes left justified in the top row 
of a $d \times (n{-}d)$ rectangle, $\lambda_2$ in the next row, etc.
For $\lambda \in \Lambda$, the number of boxes in the diagram
of $\lambda$ is denoted
$|\lambda| := \lambda_1 + \dots + \lambda_d$.  
If $|\lambda| = k$, we say $\lambda$ is a partition of $k$, and write 
$\lambda \vdash k$.  
The set $\Lambda$ is partially ordered by inclusion of diagrams: we
write $\lambda \geq \mu$ iff $\lambda_i \geq \mu_i$ for all $i$,
and $\lambda \succ \mu$ iff $\lambda > \mu$ and $|\lambda|=|\mu|+1$.

The empty partition $0 \geq \dots \geq 0$ is denoted $\varnothing$.
We denote the unique partition of $1$ by $\lowBox$, 
since its diagram consists of a single box.
The largest partition in $\Lambda$, $n{-}d \geq \dots \geq n{-}d$,
is denoted $\Rect$.

Partitions whose diagrams fit inside $\Rect$ are in bijection with 
$d$-element subsets of $\{1,\dots,n\}$:
for $\lambda \in \Lambda$, 
set 
$$J(\lambda) := \{j+\lambda_{d+1-j} \mid 1 \leq j \leq d\}\,.$$

The \bfdef{Pl\"ucker coordinates} of a point $x \in X$ are the homogeneous
coordinates $[p_\lambda(x)]_{\lambda \in \Lambda}$, defined as follows.
Suppose the subspace of $\poln$ represented by $x$ is the linear span 
of polynomials $f_1(z), \dots, f_d(z)$.
Consider the $d \times n$ matrix $A_{ij} := [z^{j-1}]f_i(z)$, whose
entries are the coefficients of the polynomials $f_i(z)$.
Then
$p_\lambda(x) := A_{J(\lambda)}$ is the maximal minor of $A$ with
column set $J(\lambda)$.

For all $\lambda \in \Lambda$, define $q_\lambda$ to be the
Vandermonde determinant
\begin{equation}
\label{eqn:defq}
q_\lambda := 
\begin{vmatrix}
1   & \cdots & 1   \\
k_1 & \cdots & k_d \\
\vdots & \vdots  & \vdots \\
k_1^{d-1} &\cdots & k_d^{d-1} \\
\end{vmatrix} = \prod_{1 \leq i<j \leq d} (k_j - k_i)\,,
\end{equation}
where $k_j = j+\lambda_{d+1-j}$.  In particular, note that $q_\lambda > 0$.

\begin{proposition}  
\label{prop:wrplucker}
The Wronskian $\Wr(x;z)$ is (up to a scalar multiple)
given explicitly in terms of the Pl\"ucker coordinates of $x$ by 
\begin{equation}
\label{eqn:pluckerwronskian}
\Wr(x;z) = \sum_{\lambda \in \Lambda}  q_\lambda p_\lambda(x) z^{|\lambda|}\,.
\end{equation}
\end{proposition}

\begin{proof}
Consider the $d \times n$ matrix $B_{ij} = (\derz)^{i-1}z^{j-1}$.
We have $(B A^t)_{ij} = f^{(i-1)}_j(z)$.
Moreover, it is not hard to calculate that the maximal minor of $B$ 
with column set $J(\lambda)$ is $q_\lambda z^{|\lambda|}$.
Thus, using the Cauchy-Binet determinant formula,
\begin{alignqed}
\Wr(x;z) &= \det (B A^t)  \\
& = \sum_{\lambda \in \Lambda} A_{J(\lambda)} B_{J(\lambda)} \\
& = \sum_{\lambda \in \Lambda} p_\lambda(x)\, q_\lambda z^{|\lambda|}\,.
\end{alignqed}
\end{proof}

\subsection{Schubert varieties}
\label{sec:schubert}

For $a \in \CP^1$, we define full flags
$$F_\bullet(a) 
= \{0\} \subset F_1(a) \subset \dots \subset F_{n-1}(a) \subset \poln$$
in $\poln$.
If $a \in \CC$, $$F_i(a) := (z+a)^{n-i}\CC[z] \cap \poln$$ 
is the set of
polynomials in $\poln$ divisible by $(z+a)^{n-i}$.
For $a = \infty$, we set
$$F_i(\infty) := \pol{i-1}\,.$$
It is straightforward to verify that 
$F_\bullet(\infty) = \lim_{a \to \infty} F_\bullet(a)$.

For every $\lambda \in \Lambda$, we have
a \bfdef{Schubert cell} relative to the 
flag $F_\bullet(a)$:
$$X^\circ_\lambda(a) := \{x \in X \mid \dim x \cap F_i(a) =
|J(\lambda) \cap \{n{-}i{+}1, \dots, n\}|\}.$$
Its closure, $X_\lambda(a) := \overline{X^\circ_\lambda(a)}$ is the
\bfdef{Schubert variety}.  The codimension of $X_\lambda(a)$ in $X$ 
is $|\lambda|$.
When the codimension is $1$, i.e. $\lambda = \lowBox$, we call 
$X_\Box(a)$ a \bfdef{Schubert divisor}.

The Schubert varieties $X_\lambda(0)$ and Schubert cells 
$X^\circ_\lambda(0)$ can be characterized in terms of the
Pl\"ucker coordinates on $X$.  

\begin{lemma}
\label{lem:pluckerschubert}
Let $x \in X$ be a closed point.  Then 
\begin{enumerate}
\item[(i)] 
$x \in X_\lambda(0)$ if and only if $p_\mu(x) = 0$ for all 
$\mu \ngeq \lambda$;
\item[(ii)]
$x \in X^\circ_\lambda(0)$ if and
only if $p_\lambda(x) \neq 0$, and $p_\mu(x) = 0$ 
for all $\mu \ngeq \lambda$.
\end{enumerate}
\end{lemma}

The proof is straightforward, using the fact that $x \in X^\circ_\lambda(0)$
iff the pivots of the matrix $A$ are in columns $J(\lambda)$.  
In fact it is true that the conditions
of Lemma~\ref{lem:pluckerschubert}(i) define $X_\lambda(0)$ 
scheme-theoretically (see \cite{HP}), but we will not need this.

\begin{theorem}
\label{thm:svroots}
Let $x \in X$ be a closed point, $a \in \CP^1$, and $k \geq 0$ an 
integer.
Then $a \in \pi(x)$ with multiplicity at least $k$ if and only if
$x \in X_\lambda(a)$ for some $\lambda \vdash k$.
\end{theorem}

\begin{proof}
By the $\SL_2(\CC)$-equivariance of the Wronski map
(Proposition~\ref{prop:sl2equivariant}), it is enough to
prove this for $a=0$.  If $x \in X_\lambda(0)$, then 
by Lemma~\ref{lem:pluckerschubert}(i), all Pl\"ucker
coordinates $p_\lambda(x)$ for $|\lambda| < k$ are zero, and hence
by~\eqref{eqn:pluckerwronskian}, $z^k$ divides $\Wr(x;z)$.  

To prove the converse, we proceed by induction.  The result is
trivially true for $k=0$; assume $0 \in \pi(x)$ has multiplicity
$k>0$, and the result is true for $k-1$.  Then $x \in X_\lambda(0)$
for some $\lambda \vdash k{-}1$; hence by Lemma~\ref{lem:pluckerschubert}(i),
$p_\mu(x) = 0$ for all
$\mu \ngeq \lambda$, in particular for all
$|\mu| \leq k{-}1$ apart from $\lambda = \mu$.  But then 
by~\eqref{eqn:pluckerwronskian},
$$\Wr(x;z) = q_\Rect p_\Rect(x) z^N + \dots 
+ q_\lambda p_\lambda(x)z^{k-1}.$$
Since $\Wr(x;z)$
is divisible by $z^k$, we see that $p_\lambda(x)=0$.  Thus by 
Lemma~\ref{lem:pluckerschubert}(ii), 
$x \notin X^\circ_\lambda(0)$.  Hence $x \in X_\lambda(0) \setminus 
X^\circ_\lambda(0)$, i.e. $x \in X_{\lambda'}$ for some 
$\lambda' > \lambda$.
\end{proof}

In particular if $\bolda = \{a_1, \dots, a_N\}$ has $N$ distinct
elements, then
$$X(\bolda) = \bigcap_{i=1}^N X_\lowBox(a_i)\,.$$
By Theorem~\ref{thm:flatfinite}, this intersection is proper, and
hence the number of intersection points counted with multiplicities 
is given by the Schubert intersection number
$$\int_X [X_\lowBox]^N\,,$$
where $[X_\lowBox] \in H^2(X)$ denotes the cohomology class of a
Schubert divisor $X_\lowBox(a)$
(which is independent of $a \in \CP^1$).  It is a basic result in 
Schubert calculus that this intersection number is the number
of standard Young tableaux of shape $\Rect$ (see e.g. \cite{Ful}).

More generally if $\bolda$ is a multiset, then set-theoretically we
have
\begin{equation}
\label{eqn:multschubertint}
X(\bolda) = \bigcap_{a \in \bolda} 
\ \bigcup_{\lambda \vdash m(a)} X_\lambda(a)\,,
\end{equation}
where $m(a)$ denotes the multiplicity of $a \in \bolda$.
However, by considering the total multiplicity of both 
sides, it is easy to see that in general this is not 
true scheme-theoretically.  For
example, if $\bolda = \{a, a, \dots, a\}$, then the right hand side
consists of the single reduced point $X_\Rect(a)$, whereas on the left
hand side this point has multiplicity $\deg \Wr = |\ordSYT(\Rect)|$.
In fact we can say more about the multiplicities in general.
Scheme theoretically,
$X(\bolda)$ is defined by 
\begin{equation}
\label{eqn:unionofschuberts}
(z+a)^{m(a)}\big|\Wr(x;z)\,,
\end{equation}
for $a \in \bolda$, which is a system of
linear equations in the Pl\"ucker variables.  In general, each
equation~\eqref{eqn:unionofschuberts} defines a non-reduced scheme
supported on a union of Schubert varieties.

\begin{corollary}
\label{cor:svrootsmultiplicities}
Let $a \in \CP^1$, and let $k$ be a positive integer.
Consider the subscheme $X(a^{(k)})$ of $X$ defined the equations 
$(z+a)^k \text{ divides } \Wr(x;z)$.
Then the cycle defined by $X(a^{(k)})$ is
$$\sum_{\lambda \vdash k} |\ordSYT(\lambda)| \cdot X_\lambda(a)\,,$$
where $\ordSYT(\lambda)$ is the number of standard Young tableaux
of shape $\lambda$.
\end{corollary}

\begin{proof}
The cycles $[X_\lambda]$ form a basis for the Chow group of $X$.
By Theorem~\ref{thm:svroots}, $X(a^{(k)})$ has support
$\bigcup_{\lambda \vdash k} X_\lambda(a)$.  Thus it is enough
to show that the multiplicity of $[X_\lambda(a)]$ in $[X(a^{(k)})]$
is $|\ordSYT(\lambda)|$.
Since $\Wr$ is flat, $X(a^{(k)})$ is rationally equivalent
to $\bigcap_{i=1}^k X_\lowBox(a_i)$ for any distinct 
$\{a_1, \dots, a_k\} \subset \CP^1$.  Thus 
$[X(a^{(k)})] =
[X_\lowBox]^k = 
\sum_{\lambda \vdash k} |\ordSYT(\lambda)| \cdot [X_\lambda(a)]$,
as required.
\end{proof}

\begin{remark} \rm
\label{rmk:ssconjparttwo}
Mukhin, Tarasov and Varchenko have recently shown \cite{MTV2} that the 
intersection on the right hand side of \eqref{eqn:multschubertint} is 
always reduced if the elements of $\bolda$ are real.  It follows from
Corollary~\ref{cor:svrootsmultiplicities} that for $\bolda \Subset \RP^1$, 
the multiplicity of a point $x \in X(\bolda)$ 
is exactly $\prod_{a \in \bolda} |\ordSYT(\lambda(x,a))|$, where 
$\lambda(x,a) \vdash m(\bolda)$ denotes the partition for which 
$x \in X_{\lambda(x,a)}(a)$.
This reducedness theorem is the second part of the Shapiro-Shapiro conjecture;
however, we will not need it in this paper.
\end{remark}

We conclude this expository section with a quick proof of
Theorem~\ref{thm:flatfinite}, using
Theorem~\ref{thm:svroots}.

\begin{proof}[Proof of Theorem~\ref{thm:flatfinite}] 
Every positive dimensional subvariety $Y$ of $X$ satisfies
$[Y]\cdot[X_\lowBox] \neq 0$ in $H^*(X)$, since $[Y]$ is a positive
linear combination of Schubert classes.  Thus if $\dim Y > 0$,
$Y \cap X_\lowBox(a) \neq \emptyset$ for all $a \in \CP^1$.

Consider a fibre $X(\bolda)$.  If
$a_0 \in \CC \setminus\bolda$, then $z+a_0$ does not divide 
$\Wr(x;z)$ for all $x \in X(\bolda)$.   
By Theorem~\ref{thm:svroots}, this means 
$X(\bolda) \cap X_\lowBox(a_0) = \emptyset$.  Thus 
$X(\bolda)$ is zero dimensional. Since $\Wr$ is projective,
this implies that it is a finite morphism.  
Flatness now follows from the the fact 
that $\Wr$ is a finite, projective morphism of non-singular 
varieties~\cite[Ch. III, Exer. 9.3(a)]{Har}.
\end{proof}


\section{Jeu de taquin theory revisited}
\label{sec:jdt}

\subsection{Standard Young tableaux with values in $\FF$}

A \bfdef{skew partition diagram} $\lambda/\mu$ is a difference of 
partition diagrams $\lambda$ and $\mu$, where $\lambda \geq \mu$.
Let $\lambda/\mu$ be a skew partition diagram which
fits inside a $d \times (n{-}d)$ rectangle, i.e. for which 
$\lambda, \mu \in \Lambda$.  
We write $\mu^c$ for the skew partition $\Rect/\mu$,
and $\mu^\vee := (n{-}d{-}\mu_d \geq \dots \geq n{-}d{-}\mu_1)$ for the partition diagram obtained by rotating $\mu^c$
by $180^\circ$.
As with partitions, $|\lambda/\mu| := |\lambda| - |\mu|$ is the number 
of boxes in $\lambda/\mu$.  

By an ordinary \bfdef{standard Young tableau} of shape $\lambda/\mu$,
we will mean the usual notion:
a filling of the boxes of $\lambda/\mu$ with entries 
$1, \dots, |\lambda/\mu|$, each used once, where the entries
increase along rows and down columns.  The set of all such tableaux
is denoted $\ordSYT(\lambda/\mu)$.
We assume some basic familiarity with the combinatorics of tableaux,
and refer the reader to~\cite{Ful}.

For our purposes, it will be convenient to have a slightly enhanced
notion of a standard Young tableau on $\lambda/\mu$.
Let $\FF$ be a field, with a norm 
$\|\cdot\| : \FF \to \RR_{\geq 0} \cup \{+\infty\}$
that is multiplicative and satisfies the triangle inequality.  
We extend $\|\cdot\|$ to $\FP^1$ by setting $\|\infty\| = +\infty$.
Let
$\bolda = \{a_1, \dots, a_{|\lambda/\mu|}\} \subset \FP^1$ be a subset of 
cardinality $|\lambda/\mu|$.  
We think of $\bolda$ as a multiset, whose elements happen to be distinct.
We impose the following restrictions, which will appear throughout this
section and Section~\ref{sec:labelling}: 
\begin{enumerate}
\item[\I] For all pairs of elements $\{a_i, a_j\}$ with 
$i \neq j$, we have $\|a_i\| \neq \|a_j\|$.
\item[\II] If $\mu \neq \varnothing$, then $0 \notin \bolda$.
\item[\III] If $\lambda \neq \Rect$, then $\infty \notin \bolda$.
\end{enumerate}
For many of our purposes $\lambda/\mu$ will be the entire rectangle 
$\Rect$, in which case restrictions~\II and~\III are irrelevant.

\begin{definition} \rm
\label{def:SYT}
A \bfdef{standard Young tableau} with \bfdef{values} in $\bolda$ and 
\bfdef{shape} $\lambda/\mu$
is a filling of the boxes of $\lambda/\mu$ with the elements of $\bolda$, 
where each element is used once and the norm of the entries is 
increasing along rows and down columns.
The set of all standard Young tableaux with values in $\bolda$ and
shape $\lambda/\mu$ is
denoted $\SYT(\lambda/\mu; \bolda)$.
\end{definition}

Let $T \in \SYT(\lambda/\mu; \bolda)$.
By replacing the smallest entry (in norm) of $T$ by $1$, the 
second smallest by $2$, and so forth, we obtain an ordinary
standard Young tableau.  We denote this tableau by 
$\ord(T) \in \ordSYT(\lambda/\mu)$.

\subsection{Sliding}
\label{sec:definesliding}

We now introduce an operation on our enhanced standard Young tableaux, 
called {\em sliding}.
To define sliding, we must assume $\FF=\RR$, with norm 
$\|\cdot\| = |\cdot|$. 

Let $T_0 \in \SYT(\lambda/\mu; \bolda_0)$.  
We can imagine $\bolda_0$ varying continuously along a path $\bolda_t$,
$t \in [0,1]$ in the space 
of $|\lambda/\mu|$-element multisubsets of $\RP^1$.  If we insist 
that $\bolda_t$ satisfy restrictions \I-\III above for all $t$, 
then $\bolda_t$ 
is in fact always a set, and there is a canonical way
to define a tableau $T_t \in \SYT(\lambda/\mu; \bolda_t)$ over the point 
$\bolda_t$, namely so 
that the entries of the family $T_t$ vary continuously, or equivalently 
so that $\ord(T_t)$ is independent of $t$.

We wish to extend this definition of $T_t$ for paths $\bolda_t$ that 
include multisets and violations of restriction~\I.  (It is
tempting to relax restrictions~\II and~\III also; unfortunately, 
this does not lead to well-behaved combinatorial structures.)
The tableau $T_t$ will not be defined at
these points of violation, but it will be defined at all other points.

First suppose $\bolda_t$, $t \in [0,1]$ is
a {\em generic}
smooth path
in the space of $|\lambda/\mu|$-element multisubsets of
$\RP^1$.  A generic path may be assumed to have the following form.
For every $t \in [0,1]$, $\bolda_t$ is a set, and 
at finitely many points $t_1, \dots, t_l \in (0,1)$ there will be a violation 
of restriction~\I of the mildest possible sort: namely,
$\bolda_{t_i} = \{a_1, \dots, a_{|\lambda/\mu|}\}$ with
$a_1 = -a_2 \notin \{0, \infty\}$, and restriction~\I holds for 
all other pairs of elements $\{a_i,a_j\} \neq \{a_1,a_2\}$.
Other sorts of violations of restriction~\I, such as multisets,
do not arise generically, as they can be avoided by perturbing the 
path (see Example~\ref{ex:positivesliding}).

In this case we define $T_t$ for $t$ near $t_i$ as follows.
If $a_1$ and $a_2$ are not in the same row or column
define $T_t$ by changing the entries continuously.
If $a_1$ and $a_2$ 
are in the same row or column define $T_t$ so that $\ord(T_t)$ is 
independent of $t$ in a neighbourhood of $t_i$.
(Note that in the former case, $\ord(T_t)$ will normally change at $t = t_i$;
in the latter case, the entries of $T_t$ will normally be discontinuous 
at $t=t_i$.)
Another way to think of
this is that $a_1$ and $a_2$ swap places if and only if they 
are forced to swap in order to maintain row and column 
strictness in the tableau.

\begin{definition} \rm
\label{def:slide}
Let
$\bolda, \bolda' \subset \RP^1$ be $|\lambda/\mu|$-element subsets
satisfying restrictions \I-\III above,
which can be joined by a path satisfying restrictions~\II and~\III.
Define 
$\slide_{\bolda'}: \SYT(\lambda/\mu; \bolda) \to \SYT(\lambda/\mu; \bolda')$
as follows.  If 
$T_0 \in \SYT(\lambda/\mu; \bolda)$, 
$\slide_{\bolda'}(T_0)$ is the tableau $T_1$ obtained by following
$T_0$ over any generic smooth path $\bolda_t$ interpolating 
$\bolda_0 = \bolda$ and $\bolda_1 = \bolda'$.
\end{definition}

\begin{theorem}
\label{thm:slidewelldefined}
The tableau $\slide_{\bolda'}(T)$ depends only on the homotopy
class of the path $\bolda_t$ in Definition~\ref{def:slide}.
\end{theorem}

\begin{proof}
This will be an immediate consequence of 
Theorem~\ref{thm:geomslide} (below) and Corollary~\ref{cor:nomonodromy}.
\end{proof}

For the main applications we consider in this paper, there will be  
additional constraints on our paths, which ensure that the homotopy
class of $\bolda_t$ in Definition~\ref{def:slide} is unique.  For this
reason, we have chosen to suppress the dependence on this homotopy class 
from our notation.  
In general, changing the homotopy class of $\bolda_t$ does have a non-trivial 
effect (see Remark~\ref{rmk:evacuation}).

In light of Theorem~\ref{thm:slidewelldefined}, 
the path $\bolda_t$ does not need to be 
generic in order to define $T_t$.  We simply put 
$T_t := \slide_{\bolda_t}(T_0)$.

\begin{example} \rm
\label{ex:positivesliding}
\newcommand{\smallskewdiagram}
{
 \begin{picture}(50,50)(3,0)
 \put(25,50){\line(1,0){25}}
 \put(0,25){\line(1,0){50}}
 \put(0,0){\line(1,0){25}}
 \put(0,0){\line(0,1){25}}
 \put(25,0){\line(0,1){50}}
 \put(50,25){\line(0,1){25}}
\end{picture}
}
Let $\bolda_t = \{\pcolor{1+2t}, \pcolor{2}\}$ for $t \in [0,1]$, and let 
$$\raisebox{20pt}{$T_0 = \ \ $}
\begin{picture}(50,50)(0,0)
 \put(25,25){\te{1}}
 \put(0,0){\te{2}}
 \put(0,0){\smallskewdiagram}
\end{picture} 
\raisebox{20pt}{\ \ .}
$$
The path $\bolda_t$ is not generic,
since $\bolda_t$ is a multiset when $t=\frac{1}{2}$;
however by perturbing the path
slightly to avoid this behaviour (see Figure~\ref{fig:perturbedpath}), 
we see that
$$T_t = \slide_{\bolda_t}(T_0) = 
\begin{cases}
\begin{picture}(50,50)(0,0)
 \put(25,25){\te{\small $1\!{+}\!2t$}}
 \put(0,0){\te{2}}
 \put(0,0){\smallskewdiagram}
\end{picture} 
& \quad \raisebox{20pt}{if $0 \leq t<\frac{1}{2}$} \medskip \\
\begin{picture}(50,50)(0,0)
 \put(25,25){\te{2}}
 \put(0,0){\te{\small $1\!{+}\!2t$}}
 \put(0,0){\smallskewdiagram}
\end{picture} 
& \quad \raisebox{20pt}{if $\frac{1}{2} < t \leq 1$\,.}
\end{cases}
$$
\begin{figure}[tbp]
\label{fig:perturbedpath}
\begin{center}
\input{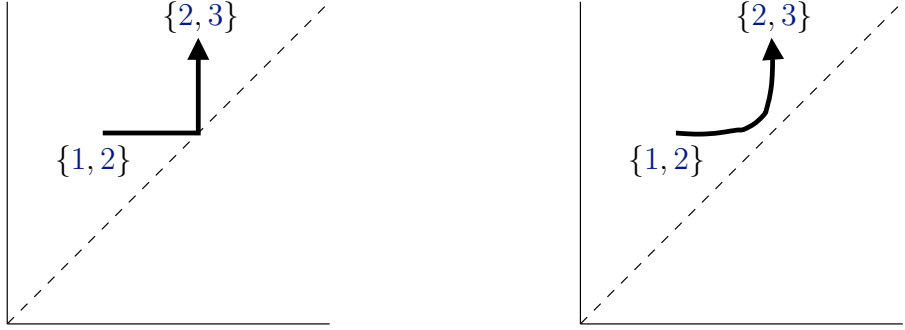}
\caption{The path $\bolda_t$ in Example~\ref{ex:positivesliding} (left),
and a slight perturbation (right).}
\end{center}
\end{figure}
Note that $\ord(T_t)$ is independent of $t$; this will always
be the case when the entries all have the same sign.  For an
illustration of the 
case with mixed signs, see Example~\ref{ex:sliding}.
\end{example}

We can now state one of our main theorems, which relates the operation of 
sliding to the Wronski map.  

\begin{theorem}
\label{thm:geomslide}
For $\bolda \subset \RR$ satisfying restriction~\I, there is a 
correspondence $x \leftrightarrow T_x$ between points 
$x \in X(\bolda)$ and tableaux 
$T_x \in \SYT(\lambda/\mu; \bolda)$.  
Under this correspondence, if $\bolda_t \Subset \RP^1$, $t \in [0,1]$
is a generic real path satisfying restrictions~\II and~\III,
and $x_t \in X(\bolda_t)$ is any lifting of $\bolda_t$ to $X$, then 
$T_{x_1} = \slide_{\bolda_1}(T_{x_0})$.  
\end{theorem}

A precise statement of the correspondence is given in 
Section~\ref{sec:labelling}, and the proof is given in 
Section~\ref{sec:monodromy}.

\subsection{Subtableaux and jeu de taquin}
\label{sec:classicjdt}

We now explain the connection between the sliding operation of
Definition~\ref{def:slide} and the usual notion of a slide
in jeu de taquin theory.

Let $\lambda/\mu$ be a skew partition, and let 
$\bolda = \{a_1, \dots, a_{|\lambda/\mu|}\} \subset \RP^1$, with
$|a_1| < \dots < |a_{|\lambda/\mu|}|$.  
Let $\boldb \subset \bolda$.
If $T \in \SYT(\lambda/\mu; \bolda)$, we denote the set of
boxes of $T$ whose entries are in $\boldb$ by $T|_\boldb$.
If $\boldb = \{a_i, a_{i+1}, \dots, a_j\}$ for some $i<j$, then 
$T|_\boldb$ is a standard Young tableau with values in $\boldb$
of some shape $\lambda'/\mu'$.  In this case we say $T|_\boldb$ is 
a \bfdef{subtableau} of $T$, and we also denote this
subtableau by $T|_{\lambda'/\mu'}$.

Let $\boldb = \{a_1, \dots, a_j\}$, and 
$\boldc = \{a_{j+1}, \dots, a_{|\lambda/\mu|}\}$.  Suppose that
all elements of $\boldb$ are positive
and that all elements in $\boldc$ are negative.
Let $a'_1, \dots, a'_j$ be positive real numbers such that
$ a'_1 > \dots > a'_j > -a_{|\lambda/\mu|}$, and set
$\boldb' = \{a'_1, \dots, a'_j\}$ and $\bolda' = \boldb' \cup \boldc$.
Note that the elements of $\boldb$ are smaller in absolute value than
the elements of $\boldc$, which are in turn smaller than those of
$\boldb'$.

In a mild abuse of notation, define
$$\slide_{T|_\boldb}(T|_\boldc) := \slide_{\bolda'}(T)|_\boldc\,.$$
By switching signs everywhere, we can also perform this construction 
if the elements of $\boldb$ are negative and the elements of $\boldc$
are positive.
Similarly, we define $\slide_{T|_\boldc}(T|_\boldb)$ 
by reversing the roles
of $\boldb$ and $\boldc$ (and reversing the inequalities) in this
construction.

Suppose $T \in \SYT(\lambda/\mu; +)
:= \bigcup_{\bolda \subset \RR_+} 
\SYT(\lambda/\mu; \bolda)$ is a standard Young tableau with all 
positive real entries
(or $T \in \SYT(\lambda/\mu; -) 
:= \bigcup_{\bolda \subset \RR_-} 
\SYT(\lambda/\mu; \bolda)$).
We can think of $\slide_T$ as an operation on tableaux which 
takes as input any skew tableau $U$ with all negative (resp. positive) 
entries that can be placed adjacent to $T$ 
to form a larger tableau,
and returns a tableau with the same entries but different shape.

If the shape of $T$ consists of a single box, it is not hard to see 
that $\slide_T(U)$ performs a Sch\"utzenberger slide or a reverse
slide through $U$ using the box of $T$ (see Example~\ref{ex:sliding}).
More generally, $\slide_T$ is 
the operation of
performing a sequence of slides in the order dictated
by the entries of $T$.  
If $T' = \slide_U(T)$, and $U' = \slide_T(U)$
then the pair $(\ord(T'), \ord(U'))$ is the result of applying
{\em tableau switching}
to the pair $(\ord(T), \ord(U))$, 
(see \cite{BSS} and the references therein).  Arguments that show that
tableau switching is well defined and independent of a number of
choices can be used to give a combinatorial proof of 
Theorem~\ref{thm:slidewelldefined}.

%
%
\begin{example} \rm
\label{ex:sliding} \rm
\newcommand{\exskewdiagram}
{
 \begin{picture}(100,75)(3,0)
 \put(25,75){\line(1,0){75}}
 \put(0,50){\line(1,0){100}}
 \put(0,25){\line(1,0){100}}
 \put(0,0){\line(1,0){50}}
 \put(0,0){\line(0,1){50}}
 \put(25,0){\line(0,1){75}}
 \put(50,0){\line(0,1){75}}
 \put(75,25){\line(0,1){50}}
 \put(100,25){\line(0,1){50}}
 \end{picture}
}
\newcommand{\exleftskewdiagram}
{
 \begin{picture}(100,75)(3,0)
 \put(25,75){\line(1,0){75}}
 \put(0,50){\line(1,0){100}}
 \put(0,25){\line(1,0){50}}
 \put(0,0){\line(1,0){25}}
 \put(0,0){\line(0,1){50}}
 \put(25,0){\line(0,1){75}}
 \put(50,25){\line(0,1){50}}
 \put(75,50){\line(0,1){25}}
 \put(100,50){\line(0,1){25}}
 \end{picture}
}
\newcommand{\exrightskewdiagram}
{
 \begin{picture}(100,75)(3,0)
 \put(50,50){\line(1,0){50}}
 \put(25,25){\line(1,0){75}}
 \put(25,0){\line(1,0){25}}
 \put(25,0){\line(0,1){25}}
 \put(50,0){\line(0,1){50}}
 \put(75,25){\line(0,1){25}}
 \put(100,25){\line(0,1){25}}
 \end{picture}
}
\newcommand{\extableau}[9]
{ 
 \begin{picture}(100,75)(0,0)
 \put(25,50){#1}
 \put(50,50){#2}
 \put(75,50){#3}
 \put(0,25){#4}
 \put(25,25){#5}
 \put(50,25){#6}
 \put(75,25){#7}
 \put(0,0){#8}
 \put(25,0){#9}
 \put(0,0){\exskewdiagram}
 \end{picture}
}
\newcommand{\exlefttableau}[6]
{ 
 \begin{picture}(100,75)(0,0)
 \put(25,50){#1}
 \put(50,50){#2}
 \put(75,50){#3}
 \put(0,25){#4}
 \put(25,25){#5}
 \put(0,0){#6}
 \put(0,0){\exleftskewdiagram}
 \end{picture}
}
\newcommand{\exrighttableau}[3]
{ 
 \begin{picture}(100,75)(0,0)
 \put(50,25){#1}
 \put(75,25){#2}
 \put(25,0){#3}
 \put(0,0){\exrightskewdiagram}
 \end{picture}
}
Let $\boldb = \{\pcolor{1},\pcolor{2},\pcolor{5}\}$, 
$\boldc = \{\ncolor{-7},\ncolor{-10},\ncolor{-13},\ncolor{-16},
\ncolor{-19},\ncolor{-22}\}$.
Let $T$ be the standard Young tableau with values in 
$\bolda = \boldb \cup \boldc$ shown below.  
$$
\raisebox{35pt}{$T =\ \ $}
\extableau                               %
          {\te{5}}  {\ue{-7}} {\ue{-16}} %
{\te{1}}  {\ue{-10}}{\ue{-13}}{\ue{-22}} %
{\te{2}}  {\ue{-19}}                     %
$$
We compute $\slide_{T|_\boldb}(T|_\boldc)$,
by increasing the entries of $\boldb$ one at a time until we reach
$\boldb' = \{\pcolor{23},\pcolor{24},\pcolor{25}\}$.  
The order in which we do this does not affect the
answer.
We choose to begin by increasing the entry $\pcolor{1}$, shown 
highlighted below.  As its value climbs past
the other positive entries in the tableau it swaps places with them, 
hence $\ord(T)$ does not change (see Example~\ref{ex:positivesliding}).
\begin{center}
\extableau                               %
          {\te{5}}  {\ue{-7}} {\ue{-16}} %
{\TE{1}}  {\ue{-10}}{\ue{-13}}{\ue{-22}} %
{\te{2}}  {\ue{-19}}                     %
$\quad \raisebox{36pt}{$\rightarrow$} \quad$
\extableau                               %
          {\te{5}}  {\ue{-7}} {\ue{-16}} %
{\te{2}}  {\ue{-10}}{\ue{-13}}{\ue{-22}} %
{\TE{3}}  {\ue{-19}}                     %
$\quad \raisebox{36pt}{$\rightarrow$} \quad$
\extableau                               %
          {\TE{6}}  {\ue{-7}} {\ue{-16}} %
{\te{2}}  {\ue{-10}}{\ue{-13}}{\ue{-22}} %
{\te{5}}  {\ue{-19}}                     %
\end{center}
As the highlighted entry continues to increase, it switches places with
the next smallest negative entry if only if the two entries are 
adjacent, thereby 
performing a Sch\"utzenberger slide through $T|_\boldc$.
\begin{center}
\extableau                               %
          {\TE{6}}  {\ue{-7}} {\ue{-16}} %
{\te{2}}  {\ue{-10}}{\ue{-13}}{\ue{-22}} %
{\te{5}}  {\ue{-19}}                     %
$\quad \raisebox{36pt}{$\rightarrow$} \quad$
\extableau                               %
          {\ue{-7}} {\TE{8}}  {\ue{-16}} %
{\te{2}}  {\ue{-10}}{\ue{-13}}{\ue{-22}} %
{\te{5}}  {\ue{-19}}                     %
$\quad \raisebox{36pt}{$\rightarrow$} \quad$
\extableau                               %
          {\ue{-7}} {\TE{11}} {\ue{-16}} %
{\te{2}}  {\ue{-10}}{\ue{-13}}{\ue{-22}} %
{\te{5}}  {\ue{-19}}                     %

\bigskip \bigskip

$\raisebox{36pt}{$\rightarrow$} \quad$
\extableau                               %
          {\ue{-7}} {\ue{-13}}{\ue{-16}} %
{\te{2}}  {\ue{-10}}{\TE{14}} {\ue{-22}} %
{\te{5}}  {\ue{-19}}                     %
$\quad \raisebox{36pt}{$\rightarrow$} \quad$
\extableau                               %
          {\ue{-7}} {\ue{-13}}{\ue{-16}} %
{\te{2}}  {\ue{-10}}{\TE{20}} {\ue{-22}} %
{\te{5}}  {\ue{-19}}                     %
$\quad \raisebox{36pt}{$\rightarrow$} \quad$
\extableau                               %
          {\ue{-7}} {\ue{-13}}{\ue{-16}} %
{\te{2}}  {\ue{-10}}{\ue{-22}}{\TE{25}}  %
{\te{5}}  {\ue{-19}}                     %
\end{center}
Next we increase the entry $\pcolor{5}$ until it is larger than $22$.
\begin{center}
\extableau                               %
          {\ue{-7}} {\ue{-13}}{\ue{-16}} %
{\te{2}}  {\ue{-10}}{\ue{-22}}{\te{25}}  %
{\TE{5}}  {\ue{-19}}                     %
$\quad \raisebox{36pt}{$\rightarrow$} \quad$
\extableau                               %
          {\ue{-7}} {\ue{-13}}{\ue{-16}} %
{\te{2}}  {\ue{-10}}{\ue{-22}}{\te{25}}  %
{\ue{-19}}{\TE{20}}                      %
$\quad \raisebox{36pt}{$\rightarrow$} \quad$
\extableau                               %
          {\ue{-7}} {\ue{-13}}{\ue{-16}} %
{\te{2}}  {\ue{-10}}{\ue{-22}}{\te{25}}  %
{\ue{-19}}{\TE{24}}                      %
\end{center}
Finally we increase the entry $\pcolor{2}$.
\begin{center}
\extableau                               %
          {\ue{-7}} {\ue{-13}}{\ue{-16}} %
{\TE{2}}  {\ue{-10}}{\ue{-22}}{\te{25}}  %
{\ue{-19}}{\te{24}}                      %
$\quad \raisebox{36pt}{$\rightarrow$} \quad$
\extableau                               %
          {\ue{-7}} {\ue{-13}}{\ue{-16}} %
{\ue{-10}}{\TE{11}} {\ue{-22}}{\te{25}}  %
{\ue{-19}}{\te{24}}                      %
$\quad \raisebox{36pt}{$\rightarrow$} \quad$
\extableau                               %
          {\ue{-7}} {\ue{-13}}{\ue{-16}} %
{\ue{-10}}{\ue{-22}}{\TE{23}} {\te{25}}  %
{\ue{-19}}{\te{24}}                      %
\end{center}
Thus we find,
$$
\raisebox{35pt}{$\slide_{T|_\boldb}(T|_\boldc) =\ \ $}
\exlefttableau                           %
          {\ue{-7}} {\ue{-13}}{\ue{-16}} %
{\ue{-10}}{\ue{-22}}                     %
{\ue{-19}}                               %
\raisebox{35pt}{\ \ .}
$$
Moreover, the relative order of the positive entries in this final tableau 
tells us,
$$
\raisebox{35pt}{
$\slide_{T|_\boldc}(T|_\boldb) =\ \ $}
\exrighttableau                          %
                    {\te{1}}  {\te{5}}   %
          {\te{2}}                       %
\raisebox{35pt}{\ \ .}
$$
\end{example}
%
%

\begin{remark} \rm
\label{rmk:evacuation}
A special case of sliding is when 
$\bolda_t = \{(a_1)_t, \dots, (a_N)_t\}$ is a loop that
cyclically rotates the elements of $\bolda_0$.  
Suppose each $(a_i)_t$ is a cyclically decreasing path in $\RP^1$,
and
$$ 0 <
(a_1)_0 = (a_2)_1 <
(a_2)_0 = (a_3)_1 
< \dots <
(a_N)_0 = (a_1)_1 \,.
$$
Let $T \in \SYT(\Rect; \bolda_0)$. By sliding $T$ using the path 
$\bolda_t$, we perform one step of Sch\"utzenberger's evacuation on 
$T$: the smallest entry performs a slide through the tableau, 
becoming the largest entry.   This procedure defines a $\ZZ$-action
on standard Young tableaux.  Since every loop is homotopic to 
some power of this basic loop, the evacuation action completely 
describes the monodromy of the sliding operation on real valued 
standard Young tableaux of shape $\Rect$.  By Theorem~\ref{thm:geomslide},
this is also the monodromy of the Wronski map for real polynomials
with $N$ or $N{-}1$ distinct real roots.
\end{remark}

\subsection{Equivalence relations on tableaux}
\label{sec:equivrelations}

We will need to adopt some additional notions from ordinary jeu de 
taquin theory.

\begin{definition} \rm
If $T \in \SYT(\lambda/\mu; \pm)$, then
the \bfdef{rectification} of $T$ is defined to be 
$\rectify(T) := \slide_{U}(T)$,
where $U \in \SYT(\mu; \mp)$ 
can be placed adjacent to $T$ to form a larger standard Young tableau.
The \bfdef{rectification shape} of $T$ is the shape of $\rectify(T)$.
\end{definition}

\begin{definition} \rm
If $T \in \SYT(\lambda/\mu; \pm)$ and $T' \in \SYT(\lambda'/\mu'; \pm)$,
we say that $T$ and $T'$ are \bfdef{equivalent}, and write
$T \sim T'$, if  $\rectify(T) = \rectify(T')$.
\end{definition}

\begin{definition} \rm
We say that $T, T' \in \SYT(\lambda/\mu; \pm)$ are \bfdef{dual equivalent},
and write $T \sim^* T'$,
if $\slide_T$ and $\slide_{T'}$ are identical as operations on
tableaux.
\end{definition}

If we replace $T, T'$ by $\ord(T), \ord(T')$, these definitions
become the usual notions of rectification, equivalence \cite{Sch}, 
and dual equivalence \cite{Hai} on standard Young tableaux.
A classical theorem of Sch\"utzenberger 
states that $\rectify(T)$ does not depend on the choice of the 
tableau $U \in \SYT(\mu; \mp)$ \cite{Sch}.

It is not hard to see that if either $T \sim T'$ or $T \sim^* T'$,
then $T$ and $T'$ have the same rectification shape.  Thus it 
makes sense to speak of the rectification shape of an equivalence
class or a dual equivalence class.
The interaction between the equivalence and dual equivalence relations
is governed by the following fact:
there is a unique tableau in the intersection of any equivalence
class of tableaux with a dual equivalence class of the same rectification
shape.  

The Littlewood-Richardson rule can be formulated in a variety
of different ways.  For us, the 
formulation below in terms of dual equivalence classes is
the most convenient.

\begin{theorem}[Littlewood-Richardson rule]
\label{thm:lrrule}
The Littlewood-Richardson number 
$$c_{\mu\nu}^\lambda := \int_X [X_{\lambda^\vee}][X_\mu][X_\nu]$$
is the number of dual equivalence classes in $\ordSYT(\lambda/\mu)$
with rectification shape $\nu$.
\end{theorem}

Alternatively, 
$c_{\mu\nu}^\lambda$ is the number of tableaux in 
$\ordSYT(\lambda/\mu)$ in any single equivalence class with 
rectification shape $\nu$.  That this statement
and Theorem~\ref{thm:lrrule} are interchangeable follows from the
relationship between equivalence and dual 
equivalence classes of
tableaux.

In Section~\ref{sec:LRrule}, we will see that rectification shape, 
equivalence, dual equivalence and many combinatorial facts
pertaining to them have natural interpretations in terms of 
the Wronski map.
Based on these, in Section~\ref{sec:lrrule} we give a new
proof of the Littlewood-Richardson rule.


\section{Labelling points in a Grassmannian by tableaux}
\label{sec:labelling}

\subsection{Fibres of the Wronski map over a non-archimedian field}

Let $\psK := \puiseux{u} = \bigcup_{n \geq 1} \CC(\!(u^{\frac{1}{n}})\!)$ 
be the field of Puiseux series over $\CC$.
In this section, we formulate a correspondence between
tableaux and 
points in the fibre of the Wronski map, working
over $\psK$.  In Section~\ref{sec:complexandreal}, we will show how this 
can be used to obtain a
correspondence over $\CC$ when the roots of the Wronskian are real. 

Let $\psX := \Gr(d, \Kpoln)$, be the Grassmannian defined over 
$\psK$.
As over $\CC$, we denote the Wronski map by
$\Wr: \psX \to \PP(\Kpol{N})$, and its fibre at 
$\prod_{a_i \neq \infty} (z+a_i)$
by $\psX(\bolda)$, where $\bolda = \{a_1, \dots, a_N\}$.
The Schubert varieties $\psX_\lambda(a)$, for $a \in \PP^1(\psK)$,
are also defined analogously.

If $g(u) = c_\ell u^\ell + \sum_{r>\ell} c_r u^r \in \psK^\times$, the 
\bfdef{valuation} of $g(u)$ is defined to be
$$\val(g(u)) := \ell\,.$$
The \bfdef{leading term} $\leadterm(g(u))$ and 
\bfdef{leading coefficient} $\leadcoeff(g(u))$
are 
$$
\leadterm(g(u)) := c_\ell u^\ell \qquad \qquad
\leadcoeff(g(u)) := [u^\ell]g(u) = c_\ell\,.$$
Additionally, we set $\val(0) := +\infty$, $\val(\infty) := -\infty$
and $\leadterm(0) := 0$.
Let $\psK_+ = \{ g(u) \in \psK \mid \val(g(u)) \geq 0\}$.

For any $0 < \varepsilon < 1$, we can define a norm on $\psK$, by
$$\|g(u)\| := \varepsilon^{\val(\tau)}\,.$$
It therefore makes sense to consider standard Young tableaux with values
in $\bolda \subset \PP^1(\psK)$.  Clearly this notion does not depend
on the choice of $\varepsilon$.  Note that in such a tableau, the
valuation of the entries {\em decreases} along rows and down columns.

Since our analysis will need to deal with cases where $\bolda$ is 
a multiset, we introduce some mild generalizations of standard Young tableaux,
called {\em weakly increasing} and {\em diagonally increasing} tableaux.
Let $\lambda/\mu$ be a skew partition fitting inside $\Rect$, and
let $\bolda  = \{a_1, \dots, a_{|\lambda/\mu|}\} \Subset \PP^1(\psK)$ 
be a $|\lambda/\mu|$-element multisubset
satisfying restrictions \II and \III, but not necessarily \I.

\begin{definition} \rm
\label{def:DIT}
A \bfdef{weakly increasing tableau} with shape $\lambda/\mu$ and
values in $\bolda$ 
is a filling of the boxes of $\lambda/\mu$
with the elements of $\bolda$ (each used as many times as its 
multiplicity) 
such that entries weakly increase in norm along
rows and down columns.
A weakly increasing tableau is \bfdef{diagonally increasing} 
if the entries are also {\em strictly} increasing in norm diagonally
right and downward.  The set of all diagonally increasing tableaux
with shape $\lambda/\mu$ and values in $\bolda$
is denoted $\DIT(\lambda/\mu; \bolda)$.
(Note that both definitions coincide with Definition~\ref{def:SYT}, 
if \I holds.)
\end{definition}

Before we can formulate the correspondence between points in $\psX$
and tableaux (Theorem~\ref{thm:leadtermequations}), 
we must introduce some notation.

For $T \in \DIT(\lambda/\mu; \bolda)$,
write 
$\val(T) := \val(a_1) + \dots + \val(a_{|\lambda/\mu|})$ for the sum of the
valuation of the entries.  
In degenerate cases where $T$ has an empty shape, $\val(T) := 0$.
The reader will note that this definition
is problematic if $0$ and $\infty$ are both entries of $T$.  
As we explain in Section~\ref{sec:zeroinfinity},
the trouble this causes is always resolvable by an appropriate
renormalization.
For now, we will state our 
results  under the assumption that $\infty \notin \bolda$.

Put
$$
\bolda^+ := \bolda \cup 
\{\underbrace{0, \dots, 0}_{|\mu|}, 
\overbrace{\infty, \dots, \infty}^{N-|\lambda|}\} \,,
$$
so that $|\bolda^+| = N$.  The reader should imagine that the extra
zeros and infinities are there to fill the boxes of $\mu$ and 
$\lambda^c$ inside $\Rect$, which do not already have entries 
from $T$ (see Theorem~\ref{thm:zeroinfinitylimit}).
Let
$$E_i(\bolda) := 
\sum_{k_1 < \dots <  k_{|\lambda/\mu|-i}} 
a_{k_1} \dotsb a_{k_{|\lambda/\mu|-i}}
$$
be the $(|\lambda/\mu|{-}i)$th elementary symmetric function, 
and put
$$e_i(\bolda) := 
[u^{\ell_i}] E_i(\bolda)
\,,$$
where $\ell_i = \min \val(a_{k_1} \dotsb a_{k_{|\lambda/\mu|-i}})$ 
is the minimum of the valuations of the terms in the sum.  Thus
$e_i(\bolda)$ equals either the leading coefficient of $E_i(\bolda)$ or $0$.

For $0 \leq i \leq |\lambda/\mu|$, define sets of partitions
$$M_i(T) := \left\{
\nu \in \Lambda \,\middle|\,%
\begin{gathered}[c]
\mu \leq \nu \leq \lambda,\quad
\nu \vdash |\mu|{+}i, \quad \text{and} \\
\val(T|_{\lambda/\nu}) \leq 
\val(T|_{\lambda/\nu'}) \text{ for all }\nu' \vdash |\mu|{+}i
\end{gathered}
\right\}\,.$$

Let $\omega_1, \dots, \omega_{|\lambda/\mu|}$ be complex variables.
Fill a skew diagram of shape $\lambda/\mu$ with entries
$\omega_1, \dots, \omega_{|\lambda/\mu|}$, 
in such a way that the position of
$\omega_i$ matches the position of $a_i$.   Let $\Omega_\nu$ 
denote the product of all
the variables $\omega_i$ which are outside of $\nu$ in this filling, 
if $\mu \leq \nu \leq \lambda$.
Put $\Omega_\nu := 0$ for all other $\nu$.  

Finally, recall the definition of $q_\nu$ from~\eqref{eqn:defq}.

\begin{theorem}
\label{thm:leadtermequations}
Let $T \in \DIT(\lambda/\mu; \bolda)$.  
Assume that $\omega_1, \dots, \omega_{|\lambda/\mu|}$ are such that
the Jacobian condition below holds:
\begin{equation}
\label{eqn:jacobiancondition}
\det J \neq 0, \quad\text{where }
J_{ij} = \frac{\partial}{\partial\omega_j} 
\sum_{\nu \in M_{i-1}(T)} q_\nu \Omega_\nu\,,
\quad i,j =1, \dots ,|\lambda/\mu|\,.
\end{equation}
There is a point $x \in \psX(\bolda^+)$  with Pl\"ucker coordinates
$[p_\nu(x)]_{\nu \in \Lambda}$ 
satisfying
\begin{equation}
\label{eqn:leadtermplucker}
\leadterm(p_\nu(x)) = \Omega_\nu u^{\val(T|_{\lambda/\nu})}
\qquad\text{for all $\nu \in \Lambda$}\,,
\end{equation}
if and only if $\omega_1, \dots, \omega_{|\lambda/\mu|}$ satisfy
\begin{equation}
\label{eqn:leadtermequations}
\sum_{\nu \in M_i(T)} q_\nu \Omega_\nu = q_\lambda e_i(\bolda)
\qquad\text{for $0 \leq i < |\lambda/\mu|$}\,.
\end{equation}
\end{theorem}

In other words, to find points in $\psX(\bolda^+)$ corresponding to 
$T \in \SYT(\lambda/\mu; \bolda)$, we solve the system of 
equations~\eqref{eqn:leadtermequations}
for $\omega_1, \dots, \omega_{|\lambda/\mu|}$, and check that the solution
satisfies~\eqref{eqn:jacobiancondition}. The following example illustrates
the details of this process.

%
%
\begin{example} \rm
\label{ex:leadtermequations} 
\newcommand{\lteskewdiagram}
{
 \begin{picture}(50,50)(3,0)
 \put(25,50){\line(1,0){25}}
 \put(0,25){\line(1,0){50}}
 \put(0,0){\line(1,0){50}}
 \put(0,0){\line(0,1){25}}
 \put(25,0){\line(0,1){50}}
 \put(50,0){\line(0,1){50}}
 \end{picture}
}
\newcommand{\ltetableau}[3]
{
 \begin{picture}(50,50)(0,0)
 \put(25,25){#1}
 \put(0,0){#2}
 \put(25,0){#3}
 \put(0,0){\lteskewdiagram}
 \end{picture} 
}
With $n = 4$, $d=2$, $\lambda= (2\geq2)$, 
$\mu = (1 \geq 0)$, and $\bolda = \{\pcolor{4u{+}2u^2}, \pcolor{1}, \pcolor{1}\}$, let
$T \in \SYT(\lambda/\mu; \bolda)$
be the tableau
$$\raisebox{20pt}{$T = \ \ $}
\ltetableau
{\te{\footnotesize $4\!u\!\raisebox{.2ex}{\tiny $+$}\!2\!u\!^2$}}
{\te{1}}
{\te{1}}
\raisebox{20pt}{\ \ .}
$$
We will apply Theorem~\ref{thm:leadtermequations} to the tableau $T$. 

First, we 
determine $e_i(\bolda)$ and $M_i(T)$ for $i =0,1,2$.
We have,
\begin{align*}
E_0(\bolda) &= 4u+2u^2 \\
E_1(\bolda) &= (4u+2u^2) + (4u+2u^2) + 1 \\
E_2(\bolda) &= (4u+2u^2) + 1 + 1\,,
\end{align*}
whence
$e_0(\bolda) = 4$, $e_1(\bolda) = 1$, $e_2(\bolda) = 2$.
Each $M_i(T)$ is a singleton: $M_i(T) = \{\alpha_i\}$, where
$$\alpha_0 = (1\geq 0), \quad \alpha_1 = (2 \geq 0), \quad 
\alpha_2 = (2 \geq 1)\,.$$
Next, we assign variables $\omega_1, \omega_2, \omega_3$ to the boxes of 
$\lambda/\mu$ as shown here
$$
\ltetableau
{\be{$\omega_1$}}
{\be{$\omega_2$}}
{\be{$\omega_3$}}
\raisebox{20pt}{\ \ ,}
$$
and write down the conditions~\eqref{eqn:jacobiancondition} 
and~\eqref{eqn:leadtermequations}.
We have
$$
q_{\alpha_0}\Omega_{\alpha_0} = 2\omega_1\omega_2\omega_3\,,\qquad
q_{\alpha_1}\Omega_{\alpha_1} = 3\omega_2\omega_3\,,\qquad
q_{\alpha_2}\Omega_{\alpha_2} = 2\omega_3\,.
$$
Thus the Jacobian matrix from~\eqref{eqn:jacobiancondition} is
$$J = 
\begin{pmatrix}
2\omega_2 \omega_3 & 2\omega_1\omega_3 & 2\omega_1\omega_2 \\
 0 & 3\omega_3 & 3\omega_2 \\
 0 & 0 & 2 
\end{pmatrix}\,,
$$
and the system of equations~\eqref{eqn:leadtermequations} is simply
\begin{align*}
2\omega_1\omega_2\omega_3 &= 4 \\
3\omega_2\omega_3 &= 1 \\
2\omega_3 &= 2\,.
\end{align*}
The solution, $\omega_1 = 6$, $\omega_2 = \frac{1}{3}$, $\omega_3 = 1$,
is a point for which $J$ is non-singular.  Therefore,
Theorem~\ref{thm:leadtermequations} asserts that there exists a point 
$x \in \psX(\bolda^+)$ whose Pl\"ucker coordinates 
satisfy~\eqref{eqn:leadtermplucker}:
\begin{equation}
\label{eqn:ltpexample}
\begin{aligned}
\leadterm(p_{0 \geq 0}(x)) &= 0 &\qquad
\leadterm(p_{2 \geq 0}(x)) &= \omega_2\omega_3 = \textstyle \frac{1}{3}  \\
\leadterm(p_{1 \geq 0}(x)) &= \omega_1\omega_2\omega_3 u = 2u &\qquad
\leadterm(p_{2 \geq 1}(x)) &= \omega_3 = 1  \\
\leadterm(p_{1 \geq 1}(x)) &= \omega_1\omega_3 u = 6u &\qquad
\leadterm(p_{2 \geq 2}(x)) &= 1 \,.
\end{aligned}
\end{equation}

A straightforward calculation shows that the two points in 
$\psX(\bolda^+)$ are 
$\langle f_1(z), f_2(z)\rangle$ and $\langle g_1(z) , g_2(z)\rangle$,
where
\begin{align*}
f_1(z) &= z^3+z^2 &
g_1(z) &= (1+u)^2z^3+(6u+3u^2)z^2 \\
f_2(z) &= z^2 + (1+u)^2z + 2u+u^2 &
g_2(z) &= z^2+ (1+u)^2z + {\textstyle \frac{1}{3}}(1+u)^2 \,.
\end{align*}
The reader can easily check that $x = \langle g_1(z), g_2(z)\rangle$ does
indeed satisfy~\eqref{eqn:ltpexample}.
\end{example}
%
%

We will prove Theorem~\ref{thm:leadtermequations} in 
Section~\ref{sec:proofleadterm}.  The most fundamental
case is when $\lambda/\mu= \Rect$ and restriction~\I holds.
In this case we obtain a bijection $x \leftrightarrow T_x$
between $\psX(\bolda)$ and $\SYT(\Rect; \bolda)$.

\begin{corollary}
\label{cor:psfibres}
Let $\bolda = \{a_1, \dots, a_N\} \subset \psK$ satisfying restriction~\I.  
For every $T \in \SYT(\Rect; \bolda)$, there is a unique point
$x_T \in \psX(\bolda)$ whose Pl\"ucker coordinates 
$[p_\nu(x_T)]_{\nu \in \Lambda}$ satisfy
\begin{equation}
\label{eqn:valplucker}
\val(p_\nu(x_T)) = \val(T|_{\nu^c})
\qquad\text{for all $\nu \in \Lambda$}\,.  
\end{equation}
Moreover, for every point $x \in \psX(\bolda)$ 
there is a unique tableau $T_x \in \SYT(\Rect; \bolda)$ such that 
$x = x_{T_x}$.  In particular, the fibre $\psX(\bolda)$ is reduced.
\end{corollary}

\begin{proof}
Assume $\|a_1\| < \dots < \|a_N\|$, and 
let $c_i := \leadcoeff(a_i)$ be the leading coefficient
of $a_i$.
Then $e_i(\bolda) = c_{i+1} \dotsb c_N$.

Let $\alpha_i$ be the shape of $T|_{\{a_1, \dots, a_i\}}$.
Then $\alpha_i$ is the unique element in $M_i(T)$, and
$\Omega_{\alpha_i} = \omega_{i+1} \dotsb \omega_N$.
The
 equations~\eqref{eqn:leadtermequations} are 
$q_{\alpha_i} \omega_{i+1} \dotsb \omega_N = q_\Rect c_{i+1} \dotsb c_N$, which has the 
unique solution 
\begin{equation}
\label{eqn:simpleomegasolution}
\omega_i = \frac{q_{\alpha_{i}}c_i}{q_{\alpha_{i-1}}}\,.
\end{equation}
At this solution, the Jacobian matrix $J$ is upper triangular, with non-zero 
entries on
the diagonal; thus~\eqref{eqn:jacobiancondition} is satisfied. 
Therefore, by Theorem~\ref{thm:leadtermequations}, 
the solution~\eqref{eqn:simpleomegasolution} gives rise to
a point $x_T$ satisfying \eqref{eqn:valplucker}.  

It is easy to 
see that if $T \neq T' \in \SYT(\Rect; \bolda)$ then 
$\val(T|_{\nu^c}) \neq \val(T'|_{\nu^c})$ for some $\nu$; 
thus we have found $|\ordSYT(\Rect)|$ distinct points in the fibre, 
which is all of them, and the uniqueness follows.
\end{proof}

\subsection{Tableau entries of $0$ and $\infty$}
\label{sec:zeroinfinity}

If $0$ is an entry of a tableau $T$, satisfying restriction~\II,
then $T$ must have a straight shape $\lambda$, and $0$ must be in the 
upper left corner.  By deleting the $0$, one obtains a skew tableau
$\tilde T$ of shape $\lambda/\lowBox$.  This new tableau $\tilde T$
produces the same equations~\eqref{eqn:leadtermplucker}
and~\eqref{eqn:leadtermequations} to be solved in 
Theorem~\ref{thm:leadtermequations}; hence $T$ and $\tilde T$
are equivalent for practical purposes in that they correspond to 
the same point(s) in $\psX$.

If $\infty$ is an entry of $T$, the situation is similar,
however we must renormalize our equations in order to make sense
of Theorem~\ref{thm:leadtermequations} and Corollary~\ref{cor:psfibres}.
For example, consider Equation~\eqref{eqn:valplucker}.
If $\infty$ is an entry of $T$, there is
a summand of $-\infty$ in each expression
$\val(T|_{\nu^c})$, except for the degenerate case $\nu = \Rect$.  
Since the Pl\"ucker coordinates are only well defined up to a 
multiplicative constant, Equation~\eqref{eqn:valplucker} should be 
regarded up to an additive constant.
If we treat $-\infty$ 
as a formal symbol, and subtract it from the valuation of each
Pl\"ucker coordinate,
we arrive at the correct replacement for \eqref{eqn:valplucker} when 
$\infty$ is an entry:
$$\val(p_\nu(x_T)) =
\begin{cases}
\val(T|_{\nu^c}{\setminus}\infty)
        &\quad \text{if $\nu \neq \Rect$} \\
+\infty &\quad \text{if $\nu = \Rect$}\,,
\end{cases}
$$
where $T|_{\nu^c}{\setminus}\infty$ means $T|_{\nu^c}$ with the box
containing $\infty$ deleted.
Other cases where $\infty$ is an entry of $T$ can be analyzed similarly,
and always one finds that the point(s) corresponding to $T$
are exactly the same as the point(s) corresponding to $T \setminus \infty$.

The next theorem further illustrates why if $\lambda/\mu \neq \Rect$, 
the boxes of
$\mu$ and $\lambda^c$ should be thought of as containing
entries of $0$ and $\infty$ respectively, for purposes of 
Theorem~\ref{thm:leadtermequations}.

\begin{theorem}
\label{thm:zeroinfinitylimit}
Let $T \in \DIT(\Rect; \bolda)$, where
$\bolda = \{a_1, \dots, a_N\} \Subset \PP^1(\psK)$ and
$$
\|a_1\| \leq \dots \leq \|a_{i-1}\| <
\|a_i\| \leq \dots \leq \|a_j\| <
\|a_{j+1}\| \leq \dots \leq \|a_N\|\,.
$$
Let $\lambda$ be the shape of $T|_{\{a_1, \dots, a_j\}}$, and let
$\mu$ be the shape of $T|_{\{a_1, \dots, a_{i-1}\}}$.
For all $t \in \psK^\times$ with $\|t\| \leq 1$, define a tableau $T_t$ 
of shape $\Rect$ obtained from $T$ as follows:  
$T_t|_{\lambda/\mu} = T|_{\lambda/\mu}$ for all $t$;
the entries $a_1, \dots, a_{i-1} \in T$ are replaced by 
$ta_1, \dots, ta_{i-1}$ in $T_t$;
the entries $a_{j+1}, \dots, a_N \in T$ are replaced by 
$t^{-1}a_{j+1}, \dots, t^{-1}a_N$ in $T_t$.

Let $x_{T_t} \in \psX$ be the point corresponding to $T_t$ (as in
Theorem~\ref{thm:leadtermequations}). Let
$x'= \lim_{t \to 0} x_{T_t} \in \psX$, 
$\bolda' = \{a_i, \dots , a_j\}$, and $T' = T|_{\lambda/\mu}$.
Then the following are true:
\begin{enumerate}
\item[(i)]  $x' \in \psX_\mu(0) \cap \psX_{\lambda^\vee}(\infty)$;
\item[(ii)]  $x'$ corresponds to $T'$;
\item[(iii)] if $\bolda'$ satisfies restriction~\I, then 
$x'$ is the unique point corresponding to $T'$.
\end{enumerate}
\end{theorem}

\begin{proof}
After normalizing the Pl\"ucker coordinates so that
$\lim_{t \to 0} p_\nu(x_{T_t})$ is defined for all $\nu \in \Lambda$,
we find that 
\begin{equation}
\label{eqn:leadtermlimit}
\leadterm(p_\nu(x')) = \lim_{t \to 0} \leadterm(p_\nu(x_{T_t})) = 
\begin{cases} 
\leadterm(p_\nu(x)) & \quad\text{if $\mu \leq \nu \leq \lambda$} \\
0 & \quad\text{otherwise.} 
\end{cases}
\end{equation}
Thus the fact that $x' \in \psX_{\mu}(0)$ follows from 
Lemma~\ref{lem:pluckerschubert}, and the fact that 
$x' \in \psX_{\lambda^\vee}(\infty)$ can be shown analogously, proving
(i).

For (ii), we must consider~\eqref{eqn:leadtermplucker} 
and~\eqref{eqn:leadtermequations} as they pertain to the 
pair $(T,x)$ and to the pair $(T', x')$.  To quell the 
notational conflicts that this naturally presents, we will use unprimed 
variable names and equation numbers ($\omega_1, \dots, \omega_N$,
\eqref{eqn:leadtermplucker}, etc.), when
referring to the context of $(T,x)$, and primed variable
names and equation numbers
($\omega'_i, \dots, \omega'_j$, \eqref{eqn:leadtermplucker}$'$, etc.) 
in the context of $(T', x')$.  

We know 
that~\eqref{eqn:leadtermplucker} holds for $(T,x)$, 
where $\omega_1, \dots, \omega_N$ 
are a solution to~\eqref{eqn:leadtermequations}.  
Equation~\eqref{eqn:leadtermlimit} gives
the lead terms of the Pl\"ucker coordinates of $x'$:  after renormalizing,
these are
\begin{equation}
\label{eqn:leadtermpluckerprime}
\leadterm(p_\nu(x')) = 
\begin{cases} 
\frac{\Omega_\nu}{\Omega_\lambda} u^{\val(T|_{\lambda/\nu})} 
& \quad\text{if $\mu \leq \nu \leq \lambda$}\,, \\
0 &\quad\text{otherwise.}
\end{cases}
\end{equation}
Set $\omega'_k = \omega_k$ for $k = i, \dots ,j$, where $\omega'_k$
is the variable corresponding to the box of $T'$ containing $a_k$.
Then $\Omega'_\nu = \Omega_\nu/\Omega_\lambda$, and so 
from~\eqref{eqn:leadtermpluckerprime} we see that~\eqref{eqn:leadtermplucker}$'$
holds for $(T', x')$.  
Furthermore,
the equations in the system~\eqref{eqn:leadtermequations} include
\begin{equation}
\label{eqn:subsetofleadtermequations}
\sum_{\nu \in M_{i+k}(T)} q_\nu \Omega_\nu = q_\Rect e_{i+k}(\bolda)
\qquad\text{for $0 \leq k < |\lambda/\mu|$}
\end{equation}
and
\begin{equation}
\label{eqn:oneleadtermequation}
q_\lambda \Omega_{\lambda} = q_\Rect e_j(\bolda)\,.
\end{equation}
We deduce that \eqref{eqn:leadtermequations}$'$ also holds for $(T',x')$ by 
dividing~\eqref{eqn:subsetofleadtermequations} 
by~\eqref{eqn:oneleadtermequation}, and
noting that $e_{i+k}(\bolda)/e_j(\bolda) = e_k(\bolda')$, 
$M_{k+i}(T) = M_k(T')$, and
$\Omega_\nu/\Omega_\lambda = \Omega'_{\nu}$. 
Since~\eqref{eqn:leadtermplucker}$'$ and~\eqref{eqn:leadtermequations}$'$
hold simultaneously, $T'$ corresponds to $x'$.

Finally, for (iii), we argue as in the proof of
Corollary~\ref{cor:psfibres}.  We have just shown that every
tableau $T \in \SYT(\lambda/\mu;\bolda')$ corresponds to
at least one point in the intersection
$\psX((\bolda')^+) \cap \psX_\mu(0) \cap \psX_{\lambda^\vee}(\infty)$.
But from Schubert calculus, we know the number of distinct points in 
this intersection is at most
$|\ordSYT(\lambda/\mu)|$, so the correspondence is bijective.
\end{proof}

\subsection{The Pl\"ucker ideal and its initial ideal}

We recall some standard facts about the equations defining $X$ and
initial ideals, for
which \cite{MS} may serve as a general reference.

Viewing $[p_\lambda]_{\lambda \in \Lambda}$ as the coordinates on
$\CP^{{n \choose d} -1}$, the Pl\"ucker coordinates define a projective
embedding of $X$; hence,
$$X = \Proj \CC[\boldp]/ I\,,$$
where $\CC[\boldp] = \CC[p_\lambda]_{\lambda \in \Lambda}$ has grading
given by $\deg p_\lambda = 1$ for all $\lambda \in \Lambda$, and
where $I$ is the \bfdef{Pl\"ucker ideal}, consisting of all polynomial
relations among the Pl\"ucker coordinates.
To state the generators of this ideal, let
$$p_{i_1, \dots, i_d} := 
\begin{cases}
\sgn(\sigma_{i_1, \dots, i_d}) p_\lambda 
&\quad
\text{if $J(\lambda) = \{i_1, \dots, i_d\}$ for some $\lambda \in \Lambda$} \\
0 &\quad\text{otherwise}\,,
\end{cases}
$$
where $\sigma_{i_1, \dots, i_d}$ denotes the permutation
that puts the list $i_1, \dots, i_d$ in increasing order.
The ideal $I$ is 
generated by all quadratics of the form
$$ \sum_{m=1}^{d+1} (-1)^m p_{i_1, \dots, i_{d-1}, j_m} \, 
p_{j_1, \dots, \widehat{j_m}, \dots, j_{d+1}}\,,$$
for $i_1, \dots, i_{d-1}, j_1, \dots, j_{d+1} \in \{1, \dots, n\}$.

Let $\boldw = (w_\lambda)_{\lambda \in \Lambda} \in \QQ^{\Lambda}$ 
be a vector of rational numbers.  The \bfdef{weight} of a monomial
$m(\boldp) = c \prod_{\lambda \in \Lambda} p_\lambda^{k_\lambda}$  
with respect
to $\boldw$ is 
$$\weight_\boldw(m) 
:= \sum_{\lambda \in \Lambda} w_\lambda k_\lambda\,.$$
Define a homomorphism $u^\boldw : \CC[\boldp] \to \psK[\boldp]$
by
$$
u^\boldw m(\boldp) 
:= u^{\weight_\boldw(m)} m(\boldp) 
$$
for monomials and extending linearly to $\CC[\boldp]$.
If $h(\boldp) \in \CC[\boldp]$, then the 
\bfdef{initial form}
of $h$ with respect to $\boldw$,
denoted $\initial_\boldw(h)$,
is the sum of all monomial
terms in $h$ for which the weight of the term is minimized.
The \bfdef{initial ideal} of the Pl\"ucker ideal $I$ with respect 
to $\boldw$ is the ideal
$$\initial_\boldw(I) 
:= \{\initial_\boldw(h) \mid h \in I\}\,.$$
In this context, 
the vector $\boldw$ is called a \bfdef{weight vector}.

The scheme $\Proj \CC[\boldp]/\initial_\boldw(I)$ can also be described
as follows.  Consider the ideal
$u^\boldw I := \{u^\boldw h \mid h \in I\}$ 
in $\CC[u^{\pm \frac{1}{\delta}}; \boldp]$.  Here the variable $u$ has
weight $0$, and $\delta$ is a common denominator of the weights 
$w_\lambda$.  Let $\tilde X$ be the closure of 
$\Proj \CC[u^{\pm \frac{1}{\delta}}; \boldp]/
(u^\boldw I \otimes \CC[u^{\pm \frac{1}{\delta}}])$ inside
$\Proj \CC[u^{\frac{1}{\delta}}; \boldp] 
= \Proj \CC[\boldp] \times \Spec \CC[u^{\frac{1}{\delta}}]$.
$\tilde X$ defines a flat family of projective varieties 
over $\Spec \CC[u^{\frac{1}{\delta}}]$, whose fibre 
at $u^{\frac{1}{\delta}} = \varepsilon$ is denoted 
$\tilde X_{\varepsilon}$.
Each of the fibres $\tilde X_\varepsilon$, $\varepsilon \neq 0$, is 
isomorphic to $X$; indeed the ring map
$h \mapsto u^{-\boldw} h|_{u^{1/\delta} = \varepsilon}$ induces
an isomorphism $\psi_\varepsilon: X \mapsto \tilde X_\varepsilon$.  
We put $\tilde X_\varepsilon(\bolda) := \psi_\varepsilon(X(\bolda))$.  
Note that $\tilde X_1$ is naturally identified with $X$.
The special fibre $\tilde X_0$ is $\Proj \CC[\boldp]/\initial_\boldw(I)$.

The same construction can be performed with 
$\CC[u^{\pm \frac{1}{\delta}}; \boldp]$ and 
$\CC[u^{\frac{1}{\delta}}; \boldp]$ replaced by 
$\psK[\boldp]$ and $\psK_+[\boldp]$ respectively.  
Note that since $u^\boldw$ acts as an automorphism
on $\psK[\boldp]$, 
$\Proj \psK[\boldp]/(u^\boldw I \otimes \psK) 
\cong \Proj \psK[\boldp]/(I \otimes \psK) = \psX$.
Its closure in $\Proj \psK_+[\boldp]$, denoted $\bar \psX$,
is a flat scheme over $\Spec \psK_+$.
Note that since $\CC[u^{\pm \frac{1}{\delta}}] \hookrightarrow \psK$, we
have a morphism $\bar \psX \to \tilde X$ , which is an isomorphism
on the fibres at $u=0$.  
Though we have suppressed it from our 
notation, the schemes $\tilde X$ and $\bar \psX$ depend on $\boldw$.

We will be primarily concerned with the case where the weight vector 
comes from a diagonally increasing tableau.
Let $T \in \DIT(\Rect; \bolda)$, where $\bolda \Subset \psK^\times$.
Then $T$ gives rise to a weight vector 
$\boldw(T) = (w_\lambda(T))_{\lambda \in \Lambda}$, where
$w_\lambda(T) := \val(T|_{\lambda^c})$, for $\lambda \in \Lambda$.

\begin{lemma}
For any $T \in \DIT(\Rect; \bolda)$, 
the initial ideal
$\initial_{\boldw(T)}(I) \subset \CC[\boldp]$ 
is generated by quadratic binomials
\begin{equation}
\label{eqn:initialideal}
p_\lambda\, p_{\lambda'} 
- p_{\lambda \vee \lambda'}\, p_{\lambda \wedge \lambda'} \,,
\end{equation}
for all $\lambda, \lambda' \in \Lambda$. 
Here, $\wedge$ and $\vee$ are the meet and join operators on 
$\Lambda$ respectively.
\end{lemma}

Thus $\boldc = [c_\lambda]_{\lambda \in \Lambda}$ represents
a point in $\tilde X_0$ if and only if 
\begin{equation}
\label{eqn:GTrelations}
c_\lambda\, c_{\lambda'} 
= c_{\lambda \vee \lambda'}\, c_{\lambda \wedge \lambda'} \,,
\end{equation}
for all $\lambda, \lambda' \in \Lambda$. 
In this case, $\tilde X_0$ is the Gel'fand-Tsetlin toric variety.


\subsection{Proof of Theorem~\ref{thm:leadtermequations}}
\label{sec:proofleadterm}

\begin{lemma}
\label{lem:aclosepoint}
For $T \in \DIT(\Rect; \bolda)$, $\bolda \Subset \psK^\times$,
let $x \in \psX$ be a point satisfying
\begin{equation}
\label{eqn:leadtermplucker2}
\leadterm(p_\nu(x)) = c_\nu u^{w_\nu(T)}
\end{equation}
for some $[c_\nu]_{\nu \in \Lambda} \in \CC^{\Lambda}$.
Then $\boldc = [c_\nu]_{\nu \in \Lambda}$ satisfies the 
relations~\eqref{eqn:GTrelations}.
Conversely, if $\boldc$ satisfies~\eqref{eqn:GTrelations},
then there is a point $x \in \psX$ for 
which~\eqref{eqn:leadtermplucker2} holds.
\end{lemma}

\begin{proof}
It is a general fact that a zero of the initial ideal over $\CC$
lifts in this way to a zero of the original ideal over $\psK$. See
\cite[Corollary 2.2]{SS}.
\end{proof}

\begin{lemma}
\label{lem:omegaform}
A point $[c_\nu]_{\nu \in \Lambda}$ is a solution to~\eqref{eqn:GTrelations} 
if and only if 
for some skew partition $\lambda/\mu$
and some $\omega_1, \dots, \omega_{|\lambda/\mu|} \in \CC^\times$,
$c_\nu = c_\lambda \Omega_\nu$ for all $\nu \in \Lambda$.
\end{lemma}

\begin{proof}
The ``if'' direction is straightforward.  For the ``only if''
direction, we note that for any solution $[c_\nu]_{\nu \in \Lambda}$ 
to~\eqref{eqn:GTrelations}, if $c_\nu \neq 0$ and $c_{\nu'} \neq 0$,
then $c_{\nu \wedge \nu'} \neq 0$ and $c_{\nu \vee \nu'} \neq 0$.
Thus the set $\{\nu \in \Lambda \mid c_\nu \neq 0\}$
has a unique maximal partition $\lambda$
and a unique minimal partition $\mu$.  
With this choice of $\lambda$ and $\mu$,
it is now straightforward to check that one can consistently define 
$\omega_i := c_{\alpha}/c_{\beta}$, where $\beta \succ \alpha$ and
the unique box of $\beta/\alpha$ corresponds to $\omega_i$.  Thus
we have
$c_\nu = c_\lambda \Omega_\nu$ for all $\nu \in \Lambda$.
\end{proof}

\begin{proof}[Proof of Theorem~\ref{thm:leadtermequations}]
First consider the case where $\lambda/\mu = \Rect$.
By Proposition~\ref{prop:wrplucker},
a point $x \in \psX(\bolda)$ is a solution to the equations:
\begin{align}
\label{eqn:satisfyplucker}
h(\boldp) = 0 \qquad \qquad \quad& \text{for }h(\boldp) \in I \\
\label{eqn:inthefibre}
\sum_{\nu \vdash k} q_\nu p_\nu = q_\Rect E_k(\bolda)
\quad\qquad& \text{for }1 \leq k \leq N\,.
\end{align}
If such an $x$ exists and satisfies~\eqref{eqn:leadtermplucker2}, 
then $c_\nu$, the leading coefficient of
$p_\nu$, is of the form $\Omega_\nu$
for some $\omega_1, \dots, \omega_N$,
by Lemma~\ref{lem:omegaform}; thus,
taking the leading term of~\eqref{eqn:inthefibre}, we find
that the equations~\eqref{eqn:leadtermequations} hold.

Conversely, suppose that we have a solution to~\eqref{eqn:leadtermequations}.
Then by Lemmas~\ref{lem:omegaform} and~\ref{lem:aclosepoint}, there is
a point $x' \in \psX$ satisfying 
\eqref{eqn:leadtermplucker}.  Thus $p_\nu = p_\nu(x')$ satisfy
\eqref{eqn:satisfyplucker}; however, the equations~\eqref{eqn:inthefibre}
are only satisfied to first order, i.e. there exists a solution to
\eqref{eqn:satisfyplucker} and
\begin{equation}
\label{eqn:implicitvariables}
\sum_{\nu \vdash k} q_\nu p_\nu = q_\Rect Y_k
\qquad \text{for }1 \leq k \leq N\,,
\end{equation}
for some $(Y_1, \dots, Y_N) \in \mathcal{U}$, where
$$
\mathcal{U} = 
\{(Y_1, \dots, Y_N) \in \psK^N \mid \val(Y_k) \geq \ell_k,\ 
[u^{\ell_k}] Y_k = e_k(\bolda)\}\,.
$$
Since $\Omega_\nu = u^{-\boldw(T)} p_\nu(x')|_{u=0}$, we can 
view $(\omega_1, \dots, \omega_N)$ as the leading coefficients of
local coordinates on $\psX$ near $x'$.
In these coordinates, the initial form of~\eqref{eqn:inthefibre}
is just~\eqref{eqn:leadtermequations}; moreover the Jacobian condition
required to apply Hensel's lemma to the system of 
equations~\eqref{eqn:implicitvariables}
is exactly~\eqref{eqn:jacobiancondition} (see e.g. \cite[Exer. 7.25]{Eis}).  
Since this Jacobian
condition is assumed to hold,
by Hensel's lemma, the points $\boldp$ 
satisfying~\eqref{eqn:satisfyplucker} and \eqref{eqn:implicitvariables}
are implicitly a function of the $Y_1, \dots, Y_N$, in the neighbourhood 
$\mathcal{U}$.  Since $(E_1(\bolda), \dots, E_N(\bolda)) \in \mathcal{U}$,
there exists a solution to~\eqref{eqn:satisfyplucker} 
and~\eqref{eqn:inthefibre}.

In the case where $\lambda/\mu \neq \Rect$, we 
consider a tableau $\tilde T \in \DIT(\Rect; \tilde \bolda)$ for which 
the $|\mu|$ smallest elements of $\tilde \bolda$ form a subtableau of
shape $\mu$, 
the $|\lambda^c|$ largest elements form a subtableau of
shape $\lambda^c$, and the remaining elements form $T$.  
Then $\tilde T|_{\lambda/\mu} = T$ and so the result follows from
Theorem~\ref{thm:zeroinfinitylimit}(ii).
\end{proof}

\subsection{Fibres of the Wronski map over $\CC$ and $\RR$}
\label{sec:complexandreal}

We now describe how one can deduce results over $\CC$ and $\RR$ from
Theorem~\ref{thm:leadtermequations} and Corollary~\ref{cor:psfibres},
which are stated over $\psK$.  We will assume implicitly here
that the solutions to~\eqref{eqn:leadtermequations} 
are always distinct (i.e. multiplicity-free), and moreover 
that~\eqref{eqn:jacobiancondition} holds for each solution.

As before, let
$\bolda = \{a_1, \dots, a_N\} \Subset \psK^\times$,
but now suppose that each 
$a_i \in \CC[u^{\pm \frac{1}{\delta}}]$
is a {\em Laurent polynomial} in some rational power of $u$.  Thus
it makes sense to evaluate $a_i$ at 
$u^\frac{1}{\delta} = \varepsilon$ for 
$\varepsilon \in \CC^\times$.  We 
denote this evaluation $a_i(\varepsilon)$, and put
$\bolda(\varepsilon) := \{a_1(\varepsilon), \dots, 
a_{|\lambda/\mu|}(\varepsilon)\}$.

We now show that for $|\varepsilon|$ sufficiently small,
we can evaluate a point $x \in \psX(\bolda)$
at $u^\frac{1}{\delta} = \varepsilon$ to obtain a point 
$x(\varepsilon) \in X(\bolda(\varepsilon))$.  If $x = x_T$ 
for $T \in \SYT(\Rect; \bolda)$, then we will declare $x(\varepsilon)$ 
to be the point corresponding to the tableau $T(\varepsilon)$, obtained
by evaluating each entry of $T$ at $\varepsilon \approx 0$.
We can make a similar declaration if $T \in \DIT(\Rect; \bolda)$,
in the cases where $T(\varepsilon)$ is actually a tableau;
however, this is less refined, as the correspondence over $\psK$
may not be one-to-one.

Let $T \in \DIT(\Rect; \bolda)$.  
By Theorem~\ref{thm:leadtermequations}, each such solution
$(\omega_1, \dots, \omega_N)$ to~\eqref{eqn:leadtermequations} 
produces a point
$x = x_T \in \psX(\bolda)$ satisfying
\eqref{eqn:leadtermplucker}.  Letting 
$p_\lambda = u^{-\boldw(T)} p_\lambda(x)$, the coordinates $[p_\lambda]$
define a point $\bar x \in \bar \psX$ over $\Spec \psK_+$.  

Since the entries $a_i$ are Laurent polynomials in $u^\frac{1}{\delta}$,
$\bar x$ is defined not just over $\psK_+$, but over a finite
algebraic extension of $\CC[u^\frac{1}{\delta}]$. This extension is 
unramified at $0$, since the solutions to~\eqref{eqn:leadtermequations} 
are distinct, and therefore $\bar x$ is 
is an analytic function of $u^\frac{1}{\delta}$ in 
some neighbourhood of $0$.  We define $\bar x(\varepsilon)$ to be the 
evaluation at this function at $u^\frac{1}{\delta} = \varepsilon$,  
and thereby obtain our point 
$x(\varepsilon) := \psi_\varepsilon^{-1}(\bar x(\varepsilon)) \in 
X(\bolda(\varepsilon))$, where
$\psi_\varepsilon$ is the isomorphism $X \to \tilde X_\varepsilon$.  

Note that
$\bar x(0)$ has coordinates $[\Omega_\lambda]_{\lambda \in \Lambda}$,
which is just the solution 
to~\eqref{eqn:leadtermequations} that we started with.
For $\varepsilon \approx 0$, $\bar x(\varepsilon) \approx \bar x(0)$.
Thus, any time we have a correspondence between points in $X(\bolda)$ 
and $\DIT(\Rect; \bolda)$ over $\psK$, we obtain a similar correspondence 
over $\CC$, wherein points in the fibre $X(\bolda(\varepsilon))$ are 
approximately described by solutions to~\eqref{eqn:leadtermequations}, 
taken over all tableaux $T \in \DIT(\Rect; \bolda)$. 
Put another way,~\eqref{eqn:leadtermplucker}
describes the asymptotic behaviour of $x(\varepsilon)$ as 
$\varepsilon \to 0$.  
Specifically,
$$p_\nu(x(\varepsilon))
\approx \Omega_\nu \varepsilon^{\delta w_\nu(T)}\,,$$
for $\varepsilon \approx 0$.

If $|\varepsilon|$ is sufficiently small, 
$\|a_i\| < \|a_j\|$ implies that $a_i(\varepsilon)$ is
of a smaller order of magnitude than $a_j(\varepsilon)$,
i.e.
$\log |a_i(\varepsilon)| \ll \log |a_j(\varepsilon)\|$.
From the proof of Corollary~\ref{cor:psfibres}, we deduce the following:
\begin{corollary}
\label{cor:complexfibres}
Let $\bolda = \{a_1, \dots, a_N\} \subset \CC$, 
with
$$ \log |a_1| \ll \dots \ll \log |a_N|\,.$$
Then every tableau $T \in \SYT(\Rect; \bolda)$ corresponds to
a point $x_T$ satisfying
$$p_\nu(x_T)
\approx \Omega_\nu\,,$$
where $\omega_i = \frac{q_{\alpha_{i-1}} a_i}{q_{\alpha_{i}}}$, and
$\alpha_i$ is the shape of $T|_{\{a_1, \dots, a_i\}}$.
\end{corollary}

\begin{corollary}
\label{cor:realfibres}
Let $\bolda = \{a_1, \dots, a_N\} \subset \RR$, with
$$ |a_1| < \dots < |a_N|\,.$$
There is a canonical bijective correspondence between tableaux 
$T \in \SYT(\Rect; \bolda)$ and points $x_T \in X(\bolda)$,
which extends the correspondence of Corollary~\ref{cor:complexfibres}.
\end{corollary}

\begin{proof}
By Corollary~\ref{cor:nomonodromy}, there is no ambiguity in
extending the correspondence, if the roots of the Wronskian 
are real.
\end{proof}

\section{Monodromy and sliding}
\label{sec:monodromy}

\subsection{When two roots have the same norm}
\label{sec:tworoots}

We now consider the case of Theorem~\ref{thm:leadtermequations} 
where $\bolda = \{a_1, \dots, a_N\} \Subset \psK^\times$, with
$$\|a_1\| < \dots < \|a_k\| = 
\|a_{k+1}\| < \dots < \|a_N\|\,.$$
From the discussion in Section~\ref{sec:complexandreal}, this
analysis will describe for us what happens to a fibre
$X(\bolda(\varepsilon))$, when two of the roots have the
same order of magnitude, while the others have different orders
of magnitude.

Let $c_i := \leadterm(a_i) u^{-\val(a_i)}$ be the leading
coefficient of $a_i$.
We have 
$$
e_i(\bolda) = 
\begin{cases}
c_{i+1} \dotsb c_N 
&\quad \text{for $i \neq k$} \\
(c_k + c_{k+1}) c_{k+2} \dotsb c_{N} 
&\quad \text{for $i = k$}\,.
\end{cases}
$$

Let $T \in \DIT(\bolda)$.  We apply
Theorem~\ref{thm:leadtermequations} to find points in the
fibre $\psX(\bolda)$ corresponding to $T$.  
There are two cases: either $a_k$ and $a_{k+1}$ are in the
same row or column of $T$, or they are in different rows and columns.

If $a_k$ and $a_{k+1}$ are in the same row or
column of $T$, then $|M_i(T)| = 1$ for all $i$.
Thus, as in the proof of Corollary~\ref{cor:psfibres}, we have
$\Omega_{\alpha_i} = \omega_{i+1} \dotsb \omega_N$, where
$\alpha_i$ is the shape of $T|_{\{a_1, \dots, a_i\}}$ and unique element 
in $M_i(T)$.
Thus the equations~\eqref{eqn:leadtermequations} become
$$
q_{\alpha_i} \omega_{i+1} \dotsb \omega_N = 
\begin{cases} 
q_\Rect c_{i+1} \dotsb c_N &\quad\text{if $i \neq k$} \\
q_\Rect (c_k + c_{k+1}) c_{k+2} \dotsb c_{N} 
&\quad\text{if $i = k$}\,.
\end{cases}
$$
If we assume that $c_k + c_{k+1} \neq 0$, 
there is a unique solution for $\omega_1, \dots, \omega_N$,
and the Jacobian condition~\eqref{eqn:jacobiancondition} 
holds at this solution.
Hence we deduce that there is a unique point $x$ in
the fibre $\psX(\bolda)$ corresponding to $T$, provided 
$c_k + c_{k+1} \neq 0$.  

Unlike in Corollary~\ref{cor:psfibres},
the correspondence is two-to-one.  The tableau $T'$ obtained by swapping
the positions of $a_k$ and $a_{k+1}$ in $T$ gives rise to the same
system of
equations, and hence also corresponds to $x$.  Thus we have a choice 
when identifying $x$ with a tableau $T_x$.  However, sometimes there
is a reason to prefer one choice over the other.  In keeping with 
the idea
that the entries of a tableau should be (weakly) increasing, 
if $\log |c_k| \ll  \log |c_{k+1}|$,
we will put $T_x = T$ if $a_k$ is above or left of $a_{k+1}$,
and $T_x = T'$ otherwise.
Similarly if $\log |c_{k+1}| \gg  \log |c_k|$,
$T_x = T$ if $a_{k+1}$ is above or left of $a_k$,
and $T_x = T'$ otherwise.

If $a_k$ and $a_{k+1}$ are in different rows and columns, there
are generally two points in the fibre corresponding to $T$, and
for a certain locus of points of $a_k, a_{k+1}$, there will be 
a double point corresponding to $T$.  We begin our analysis 
by finding this critical locus.

In this case, $|M_i(T)| = 1$ for $i \neq k$, and $M_k(T) = 2$.
Let $\alpha_i \in M_i(T)$ be the unique element for $i \neq k$,
and $M_k(T) = \{\alpha_k, \alpha_k'\}$.  We distinguish the two
elements of $M_k(T)$
by asserting that $a_k \in T|_{\alpha_k}$ 
and $a_{k+1} \in T|_{{\alpha_k'}}$.
We have 
$\Omega_{\alpha_i} = \omega_{i+1} \dotsb \omega_N$, 
$\Omega_{\alpha_k} = \omega_{k+1} \omega_{k+2}  \dotsb \omega_N$, 
$\Omega_{\alpha_k'} = \omega_k \omega_{k+2}  \dotsb \omega_{N}$. 
Thus, the system of equations~\eqref{eqn:leadtermequations} is
\begin{gather*}
q_{\alpha_i} \omega_{i+1} \dotsb \omega_N = q_\Rect c_{i+1} \dotsb c_N
\quad\text{for $i \neq k$} \\
q_{\alpha_k} \omega_{k+1} \omega_{k+2} \dotsb \omega_N +
q_{\alpha_k'} \omega_{k} \omega_{k+2} \dotsb \omega_N
= 
q_\Rect (c_k + c_{k+1}) c_{k+2} \dotsb c_{N} \,,
\end{gather*}
which in turn simplifies to 
\begin{gather}
\label{eqn:tworootseasy}
\omega_i = \frac{q_{\alpha_{i}}}{q_{\alpha_{i-1}}}c_i
\quad\text{for $i \neq k, k+1$} \\
\label{eqn:tworootssum}
\frac{q_{\alpha_k'} \omega_k + q_{\alpha_k} \omega_{k+1}}{q_{\alpha_{k+1}}}
= c_k + c_{k+1} \\
\label{eqn:tworootsprod}
\frac{q_{\alpha_{k-1}}}{q_{\alpha_{k+1}}}\, \omega_k \omega_{k+1}
= c_k c_{k+1}\,. 
\end{gather}
The equations~\eqref{eqn:tworootseasy}, give us $\omega_i$ for all 
$i \neq k, k+1$.  Solving~\eqref{eqn:tworootssum} 
and~\eqref{eqn:tworootsprod} for $\omega_k$, we find that 
\begin{equation}
\label{eqn:quadraticforomega}
\frac{q_{\alpha_k'}}{q_{\alpha_{k+1}}} 
\omega_k^2 - 
(c_k + c_{k+1}) \omega_k
+ 
\frac{q_{\alpha_k}}{q_{\alpha_{k-1}}} c_k c_{k+1}
= 0 \,.
\end{equation}
This equation has a double root when the discriminant is zero:
\begin{equation}
\label{eqn:discriminant}
(c_k + c_{k+1})^2
- 4
\frac{q_{\alpha_k}q_{\alpha_k'}}{q_{\alpha_{k-1}}q_{\alpha_{k+1}}} 
c_k c_{k+1} 
= 0\,.
\end{equation}

\begin{lemma}
\label{lem:distancecalculation}
Let $L$ be the total horizontal and vertical distances between the
two boxes in the diagram $\alpha_{k+1}/\alpha_{k-1}$.
Then
$$
\frac{q_{\alpha_k}q_{\alpha_k'}}{q_{\alpha_{k-1}}q_{\alpha_{k+1}}} 
= 1-L^{-2}
$$
\end{lemma}

\begin{proof}
Suppose the unique box of
$\alpha_{k+1}/\alpha_k' = \alpha_k/\alpha_{k-1}$ is in row $i_1$, and
the box of $\alpha_{k+1}/\alpha_k = \alpha_k'/\alpha_{k-1}$ 
is in row $i_2$.
Then 
$$L = |(\alpha_{k+1})_{i_1} - (\alpha_{k+1})_{i_2} + i_2 - i_1|\,.$$

We have
\begin{align*}
(\alpha_{k+1})_j = (\alpha_k')_j = (\alpha_k)_j &= (\alpha_{k-1})_j
\qquad \text{for $j \neq i_1, i_2$}\,, \\
(\alpha_{k+1})_{i_1} = 1+(\alpha_k')_{i_1}
= (\alpha_k)_{i_1} &= 1+(\alpha_{k-1})_{i_1}\,, \\
(\alpha_{k+1})_{i_2} = (\alpha_k')_{i_2}
= 1+(\alpha_k)_{i_2} &= 1+(\alpha_{k-1})_{i_2}\,.
\end{align*}
Thus by~\eqref{eqn:defq},
we have
\begin{align*}
\frac{q_{\alpha_k'}}{q_{\alpha_{k+1}}}
&= \prod_{j \neq d+1-i_1} 
\frac
{(j+(\alpha_k')_{d+1-j}) - (d{+}1{-}i_1 + (\alpha_k')_{i_1})}
{(j+(\alpha_{k+1})_{d+1-j}) - (d{+}1{-}i_1 + (\alpha_{k+1})_{i_1})}
\\
\frac{q_{\alpha_{k-1}}}{q_{\alpha_k}}
&= \prod_{j \neq d+1-i_1} 
\frac
{(j+(\alpha_{k-1})_{d+1-j}) - (d{+}1{-}i_1 - (\alpha_{k-1})_{i_1})} 
{(j+(\alpha_k)_{d+1-j}) - (d{+}1{-}i_1 + (\alpha_k)_{i_1})}
\,.
\end{align*}
For $j \neq d+1-i_2$, the terms in these two products are equal.  
Thus,
\begin{alignqed}
\frac{q_{\alpha_k}q_{\alpha'_k}}{q_{\alpha_{k-1}}q_{\alpha_{k+1}}} 
&= 
\frac
{(-i_2+(\alpha_k)_{i_2} + i_1 - (\alpha_k)_{i_1})
(-i_2+(\alpha'_k)_{i_2} + i_1 - (\alpha'_k)_{i_1}) 
}
{(-i_2+(\alpha_{k-1})_{i_2} + i_1 - (\alpha_{k-1})_{i_1})
(-i_2+(\alpha_{k+1})_{i_2} + i_1 - (\alpha_{k+1})_{i_1})
}
\\
&=
\frac
{(-i_2 \spc+ (\alpha_{k+1})_{i_2} \spc+ i_1 \spc- (\alpha_{k+1})_{i_1}{+}1)
(-i_2 \spc+ (\alpha_{k+1})_{i_2}{-}1 \spc+ i_1 \spc- (\alpha_{k+1})_{i_1})
}
{(-i_2+(\alpha_{k+1})_{i_2} + i_1 - (\alpha_{k+1})_{i_1})
(-i_2+(\alpha_{k+1})_{i_2} + i_1 - (\alpha_{k+1})_{i_1})
} 
\\
&=
\frac{(L-1)(L+1)}{L^2} 
\\
&= 1- L^{-2}\,. 
\end{alignqed}
\end{proof}

\begin{lemma}
\label{lem:jacobianholds}
The discriminant of~\eqref{eqn:quadraticforomega} is non-zero
if and only if the 
solutions to~\eqref{eqn:tworootseasy}--\eqref{eqn:tworootsprod}
are a point at which the Jacobian condition~\eqref{eqn:jacobiancondition} 
holds.
\end{lemma}

\begin{proof}
The matrix Jacobian matrix $J$ of \eqref{eqn:jacobiancondition} is block
upper triangular, with all diagonal blocks non-zero of size $1 \times 1$, 
except for a $2 \times 2$ block in rows $k, k{+}1$.  
Thus~\eqref{eqn:jacobiancondition} holds iff the determinant of this
$2 \times 2$ block
$$
\begin{pmatrix}
\frac{\partial}{\partial\omega_k} 
q_{\alpha_{k-1}} \Omega_{\alpha_{k-1}} 
& \frac{\partial}{\partial\omega_{k+1}} 
q_{\alpha_{k-1}} \Omega_{\alpha_{k-1}} 
\\
\frac{\partial}{\partial\omega_k} 
(q_{\alpha_k} \Omega_{\alpha_k} + q_{\alpha'_k} \Omega_{\alpha'_k})
& \frac{\partial}{\partial\omega_{k+1}} 
(q_{\alpha_k} \Omega_{\alpha_k} + q_{\alpha'_k} \Omega_{\alpha'_k})
 \\
\end{pmatrix}
=
\Omega_{\alpha_{k+1}}
\begin{pmatrix}
q_{\alpha_{k-1}} \omega_{k+1}
& 
q_{\alpha_{k-1}} \omega_{k}
\\
q_{\alpha_k'} 
& q_{\alpha_k} 
\end{pmatrix}
$$
is non-zero, i.e. 
iff $q_{\alpha_k}\omega_{k+1} \neq q_{\alpha_k'}\omega_{k}$.

On the other hand, if $(\omega_k, \omega_{k+1})$ is one solution 
to~\eqref{eqn:tworootssum} and~\eqref{eqn:tworootsprod},
then the other solution is
$(q_{\alpha_k}\omega_{k+1}/q_{\alpha_k'},
q_{\alpha_k'}\omega_{k}/q_{\alpha_k})$.
The discriminant of~\eqref{eqn:quadraticforomega} is non-zero iff
these two solutions are distinct, i.e. iff
$q_{\alpha_k}\omega_{k+1} \neq q_{\alpha_k'}\omega_{k}$.
\end{proof}

\begin{corollary}
\label{cor:tworootsreal}
If $c_k, c_{k+1} \in \RR$, then the system of 
equations~\eqref{eqn:tworootssum} and~\eqref{eqn:tworootsprod} 
has two distinct real solutions, hence there are two points in 
$\psX(\bolda)$ corresponding to $T$, i.e.
satisfying~\eqref{eqn:leadtermplucker}.
\end{corollary}

\begin{proof}
It is enough to check that the discriminant of~\eqref{eqn:quadraticforomega}
is positive.  Since $q_\lambda > 0$ for all $\lambda \in \Lambda$, this
is certainly true if $c_k c_{k+1} < 0$.  Otherwise, we
have
\begin{align*}
(c_k + c_{k+1})^2
- 4
\frac{q_{\alpha_k}q_{\alpha_k'}}{q_{\alpha_{k-1}}q_{\alpha_{k+1}}} 
c_k c_{k+1}
&= 
(c_k + c_{k+1})^2
- 4 (1-L^{-2}) c_k c_{k+1} \\
& >
(c_k + c_{k+1})^2 - 4 c_k c_{k+1} \\
& =
(c_k - c_{k+1})^2 \\
& \geq 0\,.
\end{align*}
By Lemma~\ref{lem:jacobianholds},
we can apply Theorem~\ref{thm:leadtermequations} to conclude that
we have two corresponding points in the fibre $\psX(\bolda)$.
\end{proof}

The reason $T$ is identified with two points in $\psX(\bolda)$
rather than one
is that there is a tie in the order of magnitude of the roots.
As in the same-row/column case, the 
tableau $T' \in \DIT(\Rect; \bolda)$, obtained
by swapping the positions of $a_k$ and $a_{k+1}$ in $T$,
produces the same system of equations, and
hence is also identified with these same two points.
Thus, we have a two-to-two correspondence between tableaux in
$\DIT(\Rect;\bolda)$ and points in $\psX(\bolda)$.
Note that between this two-to-two correspondence, and the two-to-one
correspondence earlier, we have found all $|\ordSYT(\Rect)|$
points in $\psX(\bolda)$.

Now suppose that $\log |c_k| \ll \log |c_{k+1}|$.  This supposition
effectively breaks the tie in the order of magnitude of the roots, 
which gives a natural way to
identify $T$ with one of these two points in $\psX(\bolda)$, and
$T'$ with the other.
To see this, we put
$c_k = \bar u^{v_1} b_1$ and $c_{k+1} = \bar u^{v_2} b_2$,
with $v_1 > v_2$,
and solve
\eqref{eqn:tworootssum} and~\eqref{eqn:tworootsprod} over 
$\puiseux{\bar u}$.

\begin{proposition}
If 
$c_k = \bar u^{v_1} b_1$ and $c_{k+1} = \bar u^{v_2} b_2$,
and $v_1 > v_2$,
then the one solution for $(\omega_k, \omega_{k+1})$ satisfies
\begin{equation}
\label{eqn:firstasymptoticsolution}
\leadterm(\omega_k) 
 = \frac{q_{\alpha_{k}}c_{k}}{q_{\alpha_{k-1}}} \qquad\qquad
\leadterm(\omega_{k+1}) 
 = \frac{q_{\alpha_{k+1}}c_{k+1}}{q_{\alpha_{k}}}\,,
\end{equation}
and the other satisfies
\begin{equation}
\label{eqn:secondasymptoticsolution}
\leadterm(\omega_{k+1}) 
 = \frac{q_{\alpha_k'}c_{k}}{q_{\alpha_{k-1}}} \qquad\qquad
\leadterm(\omega_k) 
 = \frac{q_{\alpha_{k+1}}c_{k+1}}{q_{\alpha_k'}}\,.
\end{equation}
\end{proposition}

\begin{proof}
By Hensel's lemma, there exists a solution for $y_1, y_2 \in \psK_+$
to the system of equations
\begin{gather*}
\frac{\bar u^{v_1-v_2}q_{\alpha_k'}\,y_1 + q_{\alpha_k}y_2}{q_{\alpha_{k+1}}}
=\bar u^{v_1-v_2} b_1 +  b_2 \\
\frac{q_{\alpha_{k-1}}}{q_{\alpha_{k+1}}}\, y_1 y_2
= b_1 b_2 
\end{gather*}
with 
$$
\leadterm(y_1)
 = \frac{q_{\alpha_{k}}b_{1}}{q_{\alpha_{k-1}}} \qquad\qquad
\leadterm(y_2) 
 = \frac{q_{\alpha_{k+1}}b_{2}}{q_{\alpha_{k}}}\,.
$$
Putting $\omega_k = \bar u^{v_1} y_1$, $\omega_{k+1} = \bar u^{v_2} y_2$
gives the first solution.  The other solution is obtained similarly.
\end{proof}

Replacing $\omega_k, \omega_{k+1}$ by 
$\leadterm(\omega_k), \leadterm(\omega_{k+1})$, these are
precisely the solutions to two different systems of
equations~\eqref{eqn:leadtermequations} that we obtain if we
perturb the norm of the entries so that
$\|a_k\| \neq \|a_{k+1}\|$.
The first solution~\eqref{eqn:firstasymptoticsolution} is the one
that is consistent with breaking the tie so that
$\|a_1\| < \dots < \|a_k\| < \|a_{k+1}\| < \dots < \|a_N\|$.  To
see this, note that if $\|a_k\| < \|a_{k+1}\|$, then 
$M_k(T) = \{\alpha_k\}$; 
thus, as in the proof of Corollary~\ref{cor:psfibres},
the solution 
to~\eqref{eqn:leadtermequations} is given 
by~\eqref{eqn:simpleomegasolution}, which is consistent 
with~\eqref{eqn:firstasymptoticsolution}.
Since $\log |c_k| \ll \log |c_{k+1}|$, this is the solution 
we identify with $T$.
The second solution~\eqref{eqn:secondasymptoticsolution} corresponds 
to 
$\|a_1\| < \dots < \|a_{k-1}\| < \|a_{k+1}\| < \|a_k\| < \dots < \|a_N\|$,
since here $M_k(T) = \{\alpha_{k+1}\}$.
This solution is identified with the other tableau, $T'$.

In summary, suppose that either $\log |c_k| \ll \log |c_{k+1}|$ or
$\log |c_k| \gg \log |c_{k+1}|$.  Then
a solution $(\omega_k, \omega_{k+1})$ 
to~\eqref{eqn:tworootssum} and~\eqref{eqn:tworootsprod}
is identified with the tableau $T$ for which
$\omega_k \approx c_k$ and $\omega_{k+1} \approx c_{k+1}$
(up to a ratio of $q_\alpha$s).  If $(\hat \omega_k, \hat \omega_{k+1})$
denotes the other solution to~\eqref{eqn:tworootssum} 
and~\eqref{eqn:tworootsprod} then 
$\hat \omega_k \approx c_{k+1}$ and $\hat \omega_{k+1} \approx c_k$,
and this solution corresponds to $T'$.

\subsection{Proof of Theorem~\ref{thm:geomslide}}

Suppose that $\bolda_t = \{(a_1)_t, \dots, (a_N)_t\}$, 
$t \in [0,1]$ is a path in the space of multisubsets of $\PP^1(\psK)$,
with $\bolda_0 = \bolda$ as in Section~\ref{sec:tworoots}, 
and for all $t\in [0,1]$,
$(a_i)_t = a_i$ if $i \neq k, k+1$, and
$\|a_k\| = \|(a_k)_t\| = \|(a_{k+1})_t\| = \|a_{k+1}\|$.
Let $(c_i)_t := \leadcoeff((a_i)_t)$, 
and suppose $(c_k)_t$ and $(c_{k+1})_t$, $t \in [0,1]$
are paths in $\RR^\times$.  
Let 
$x_t \in \psX(\bolda_t)$ be a path in $\psX$. 
Finally, suppose 
$\log|(c_k)_0| \ll \log|(c_{k+1})_0|$ and
$\log|(c_k)_1| \gg \log|(c_{k+1})_1|$.
From the discussion in Section~\ref{sec:tworoots}, these hypotheses
imply that there are unique tableaux $T_{x_0}$ and $T_{x_1}$ 
corresponding to points $x_0$ and $x_1$.  Since this correspondence
is defined asymptotically, for other values 
of $t \in (0,1)$
we do not associate a unique corresponding tableau $T_{x_t}$.

\begin{theorem}
\label{thm:simpleslide}
With $\bolda_t$ and $x_t$, as above, $T_{x_0}$ and $T_{x_1}$ are related
as follows. 
\begin{enumerate}
\item[(i)]  If $(a_k)_0$ and $(a_{k+1})_0$ are in the same row or
column of $T_{x_0}$, or if $(c_k)_0\,(c_{k+1})_0 > 0$, then 
$T_{x_1}$ is obtained from $T_{x_0}$ by replacing
$(a_k)_0$ with $(a_{k+1})_1$ and $(a_{k+1})_0$ with $(a_k)_1$.
\item[(ii)]  If $(a_k)_0$ and $(a_{k+1})_0$ are in different rows and
columns of $T_{x_0}$ and $(c_k)_0\,(c_{k+1})_0 < 0$, then 
$T_{x_1}$ is obtained from $T_{x_0}$ by replacing
$(a_k)_0$ with $(a_k)_1$ and $(a_{k+1})_0$ with $(a_{k+1})_1$.
\end{enumerate}
\end{theorem}

\begin{proof}
Let $T_0 = T_{x_0}$, and $T_t$ be the tableau obtained from $T_0$
by replacing $(a_i)_0$ with $(a_i)_t$ for all $i$.
Let $T'_t$ be the tableau obtained by swapping the positions
of $(a_k)_t$ and $(a_{k+1})_t$ in $T_t$.
The point $x_t$ satisfies the conditions of 
Theorem~\ref{thm:leadtermequations} for both tableaux
$T_t$ and $T'_t$.
Thus we must have either $T_{x_1} = T_1$ or $T_{x_1} = T'_1$.

If $(a_k)_0$, $(a_{k+1})_0$ are in the same row or column of $T_{x_0}$,
then $T_{x_1} = T'_1$ simply by definition.  There is one small problem,
however, which is that we have only established the existence of a
point $x_t$ provided $(c_k)_t + (c_{k+1})_t \neq 0$.  
To get around this, note that Theorem~\ref{thm:ssconj}
guarantees that the fibre 
$\psX(\bolda_t)$ is reduced even if $(c_k)_t + (c_{k+1})_t = 0$. 
Thus $T_{x_1}$ is unaffected
by small perturbations of the path $\bolda_t$.  
We can therefore 
perturb the path $\bolda_t$ so that $(c_k)_t$ and $(c_{k+1})_t$ become 
complex paths 
such that $(c_k)_t + (c_{k+1})_t \neq 0$ for all $t$, and hence see 
that $T_{x_1} = T'_1$.
This establishes the first case of (i).

For the remaining cases, suppose that 
$(a_k)_0$, $(a_{k+1})_0$ are in different rows and columns of $T_{x_0}$.
Let $((\omega_k)_t, (\omega_{k+1})_t)$ be the 
solution to~\eqref{eqn:tworootssum} and~\eqref{eqn:tworootsprod}
which gives rise to the point $x_t \in \psX(\bolda_t)$ 
via~\eqref{eqn:leadtermplucker},
and let
$((\hat \omega_k)_t, (\hat \omega_{k+1})_t)$ be the second solution
to these equations.  
In each case, we will need to determine whether 
$x_1$ corresponds to $T_1$ or $T'_1$.  From the discussion at the end of
Section~\ref{sec:tworoots}, if $x_1$ corresponds to $T_1$, then
$(\omega_k)_1 \approx (c_k)_1$ and $(\omega_{k+1})_1 \approx (c_{k+1})_1$.
If $x_1$ corresponds to $T'_1$, then 
$(\omega_{k+1})_1 \approx (c_k)_1$ and
$(\omega_k)_1 \approx (c_{k+1})_1$.

Suppose that $(c_k)_0 > 0$ and $(c_{k+1})_0 > 0$.
Since $x_0$ 
we have 
$\log(\omega_k)_0 \approx \log(c_k)_0$ and
$\log(\hat \omega_k)_0 \approx \log(c_{k+1})_0$.
Since $\log (c_k)_0 \ll \log (c_{k+1})_0$,
It follows that
$\log(\omega_k)_0 \ll \log(\hat \omega_k)_0$.  
By Corollary~\ref{cor:tworootsreal},  
$(\omega_k)_t \neq (\hat \omega_k)_t$ for all $t \in [0,1]$;
thus 
$(\omega_k)_t > (\hat \omega_k)_t > 0$ for all $t \in [0,1]$.
Since 
$\log (c_{k+1})_1 \gg \log (c_k)_1$, it must be the case that
$(\omega_k)_1 \approx (c_{k+1})_1$ and 
$(\hat \omega_k)_1 \approx (c_k)_1$, rather than the other 
way around.  
Thus we see that
$T_{x_1} = T'_1$.
Similarly, we have $T_{x_1} = T'_1$
if $(c_k)_0 < 0$ and $(c_{k+1})_0 < 0$.

Now suppose  $(c_k)_0 > 0$ and $(c_{k+1})_0 < 0$. 
Then we must also have
$(\omega_k)_0  \approx (c_k)_0 > 0$.
Since $(c_k)_t \, (c_{k+1})_t \neq 0$ for all
$t$, by~\eqref{eqn:tworootsprod}, the signs of 
$(\omega_k)_t$
and $(c_{k+1})_t$ 
are independent of $t$.
In particular, $(\omega_k)_1$ is positive, while $(c_{k+1})_1$
is negative.  Since these have opposite signs, it cannot be
the case that $(\omega_k)_1 \approx (c_{k+1})_1$, hence ${x_1}$ is not
identified with $T_1'$.
We must therefore have $T_{x_1} = T_1$.
Similarly, $T_{x_1} = T_1$
if $(c_k)_0 < 0$ and $(c_{k+1})_0 > 0$.
\end{proof}

Theorem~\ref{thm:geomslide} now follows.

\begin{proof}[Proof of Theorem~\ref{thm:geomslide}]
Let
$\bolda_t = \{(a_1)_t, \dots, (a_N)_t\} \subset \RP^1$, 
$t \in [0,1]$ is a path in the space of multisubsets of $\RP^1$.

First, consider the case where
$$ \log |(a_1)_t| \ll \dots \ll \log|(a_k)_t|\,,\, \log |(a_{k+1})_t|
\ll \dots \ll \log |(a_N)_t|\,,$$
$\log |(a_k)_0| \ll \log|(a_{k+1})_0|$
and $\log |(a_{k+1})_1| \ll \log|(a_k)_1|$.
Let $x_t \in X(\bolda_t)$.  Then by
Theorem~\ref{thm:simpleslide} and
the discussion in Section~\ref{sec:complexandreal} we see that
$T_{x_1} = \slide_{a_1}(T_{x_0})$.

Second, suppose that
$$ |(a_1)_t| <  \dots < |(a_k)_t|\,,\, |(a_{k+1})_t| < \dots < |(a_N)_t|\,,$$
$|(a_k)_0| < |(a_{k+1})_0|$
and $|(a_{k+1})_1| < |(a_k)_1|$.
There is an order and sign preserving homotopy between this case and
the previous.  Since the correspondence of Corollary~\ref{cor:realfibres}
between points in $X(\bolda)$
and $\SYT(\Rect;\bolda)$ is established via such homotopies, 
the theorem holds in 
this case also.

Finally, a general path $\bolda_t \Subset \RP^1$ can be regarded as
a concatenation of paths from the second case; thus the theorem is true 
for any real path.
\end{proof}

\subsection{Monodromy around special loops}

Let $\lambda/\mu$ be a skew partition fitting inside $\Rect$.
Throughout the rest of this section, we will assume that $k,L$ are positive
integers with $1 \leq k < |\lambda/\mu|$, and $L \geq 2$.

For any such $k, L$,
define a permutation $s_{k,L}$ of the set $\ordSYT(\lambda/\mu)$
as follows.  
For $\ordT \in \ordSYT(\lambda/\mu)$, $s_{k,L}(\ordT)$ is the tableau
obtained by swapping entries $k$ and $k{+}1$ in $\ordT$, if the total of
the horizontal and vertical distance between $k$ and $k{+}1$ equals
$L$; otherwise $s_{k,L}(\ordT) = \ordT$.

If $\bolda = \{a_1, \dots, a_{|\lambda/\mu|}\} \subset \FP^1$, 
then we define
$s_{k,L}(T)$ for $T \in \SYT(\lambda/\mu; \bolda)$, to
satisfy $\ord(s_{k,L}(T)) = s_{k,L}(\ord(T))$.
If $\bolda \subset \RR$, we
also define $s_{k,L}(x)$ for $x \in X(\bolda^+)$, by
$T_{s_{k,L}(x)} = s_{k,L}(T_x)$.

\begin{theorem}
\label{thm:monodromyloops}
Fix $k$ and $L$ as above, and
let $\bolda = \{a_1, \dots, a_{|\lambda/\mu|}\} \subset \RP^1$.
There exists a loop $\bolda_t \subset \CC$, $t \in [0,1]$, based at
$\bolda$
such that the monodromy of the Wronski map around $\bolda_t$
is given by $s_{k,L}$.  That is,
$\bolda_0 = \bolda_1 = \bolda$, every fibre
$X(\bolda_t)$ is reduced, and if $x_t \in X(\bolda_t)$ then
$x_1 = s_{k,L}(x_0)$.
\end{theorem}

\begin{proof}
By the discussion in Section~\ref{sec:complexandreal}, it is enough to 
prove the result over $\psK$, with
$\bolda$ as in Section~\ref{sec:tworoots}.

Consider a loop $\bolda_t = \{(a_1)_t, \dots, (a_N)_t\}$, 
$t \in [0,1]$ in the space of multisubsets of $\PP^1(\psK)$,
with $\bolda_0 = \bolda_1 = \bolda$, and for all $t\in [0,1]$,
$(a_i)_t = a_i$ if $i \neq k, k+1$, 
and $\|a_k\| = \|(a_k)_t\| = \|(a_{k+1})_t\| = \|a_{k+1}\|$.
Let $(c_i)_t := \leadcoeff((a_i)_t)$.
Suppose that $((c_k)_t, (c_{k+1})_t) \in \CC^2$ is a small loop around 
the line
\begin{equation}
\label{eqn:criticalline}
\{(c_k, c_{k+1}) \mid c_k = \big(1+2L^{-2} + 2L^{-1}\sqrt{1+L^{-2}}\big) c_{k+1}\}\,.
\end{equation}
We show that if $x_t \in \psX(\bolda_t)$, then $x_1 = s_{k,L}(x_0)$.

Recall that a tableau $T_t \in \DIT(\Rect; \bolda_t)$ corresponds 
to one or two points.  If the distance between $a_k$ and $a_{k+1}$
is $1$ then these entries are in the same row or column, 
so $T_{x_t}$ corresponds 
to the single point $x_t$; hence $x_1 = x_0 = s_{k,L}(x_0)$.  Otherwise,
$T_{x_t}$ corresponds to $x_t$ and another point.  Since the equations
that give the leading terms of these two points,
\eqref{eqn:tworootssum} and~\eqref{eqn:tworootsprod}, define a quadratic
map
$$(\omega_k, \omega_{k+1}) \mapsto 
\left(
\frac{q_{\alpha_k'} \omega_k + q_{\alpha_k} \omega_{k+1}}{q_{\alpha_{k+1}}},
\frac{q_{\alpha_{k-1}}}{q_{\alpha_{k+1}}}\, \omega_k \omega_{k+1}
\right)\,,
$$
the two points
will swap places under the monodromy of the loop $\bolda_t$ if and only 
if the leading coefficients of $\bolda_t$ 
wrap around the critical locus~\eqref{eqn:discriminant}.  By
Lemma~\ref{lem:distancecalculation}, the line~\eqref{eqn:criticalline}
is contained in the critical locus iff the distance between 
$(a_k)_t$ and $(a_{k+1})_t$ is $L$.
\end{proof}

\subsection{Monodromy and limits}
\label{sec:limits}

Let $\bolda  = \{a_1, \dots, a_N\} \subset \RP^1$, with
$|a_1|< \dots < |a_N|$, and let
$x \in X(\bolda)$.

\begin{definition} \rm
Suppose that we have a decomposition of $\RP^1$ as the disjoint union 
of $k$ intervals $I_1, \dots, I_k$.
We then obtain a partition $(\boldb_1, \dots , \boldb_k)$
of $\bolda$, where $\boldb_i = \bolda \cap I_i$.  A partition of $\bolda$ 
of this form is called a \bfdef{consecutive partition} of $\bolda$.
Any number $c_i \in I_i$ is called 
an \bfdef{internal point} for $\boldb_i$.
\end{definition}

The two main cases we will consider are given in the
examples below.

\begin{example} \rm
\label{ex:zeroinfinitycase}
For any $\bolda$ as above,
let $\boldb_0 = \{a_1, \dots, a_i\}$, $\boldb_\infty = \{a_{i+1}, \dots, a_N\}$
for some $i$.
Then $(\boldb_0, \boldb_\infty)$ is a consecutive partition of $\bolda$.
Moreover, for any tableau $T \in \SYT(\Rect; \bolda)$,
$T|_{\boldb_0}$ and $T|_{\boldb_\infty}$ are both subtableaux of $T$.
\end{example}

\begin{example} \rm
\label{ex:tableaucase}
Let $\boldb = \{a_i, a_{i+1}, \dots, a_j\}$, 
$\boldb^c = \bolda \setminus \boldb$,
for some $i\leq j$.
Suppose that all elements of $\boldb$ have the same sign. Then
$(\boldb, \boldb^c)$ is a consecutive partition of $\bolda$.
In this case, for any tableau $T \in \SYT(\Rect; \bolda)$,
$T|_\boldb$ is a subtableau of $T$.
\end{example}

Suppose we have a consecutive partition $(\boldb_1, \dots, \boldb_k)$ of
$\bolda$,  coming from intervals $I_1, \dots, I_k \subset \RP^1$. 
Let $c_i$ be an internal point for $b_i$.
We define points $x_{[\boldb_i{\to}c_i]} \in X$, as follows.

For each fixed $i$, we form a path
$\bolda_t \Subset \RP^1$, $t \in [0,1]$ satisfying the following
conditions:

\bigskip
\begin{minipage}{2.6in}
\begin{enumerate}
\item[(i)] 
$\bolda_0 = \bolda$;
\item[(ii)] 
$\bolda_t$ is a set for $t \in [0,1)$;
\end{enumerate}
\end{minipage}
\begin{minipage}{2.6in}
\begin{enumerate}
\item[(iii)] 
$\bolda_t \cap I_j= \boldb_j$, for $j \neq i$, $t \in [0,1]$;
\item[(iv)] 
$\bolda_1 = \big(\bigcup_{j \neq i} \boldb_j \big) 
\cup \{c_i, c_i, \dots, c_i\}$.
\end{enumerate}
\end{minipage}

\bigskip

Let $x_0 = x$, and $x_t \in X(\bolda_t)$.  Since $\bolda_t$ is
a set for $x_t \in [0,1)$ there is a unique such path for 
$t \in [0,1)$.  
We define
$x_{[\boldb_i \to c_i]}$ to be the limit point $x_1 = \lim_{t \to 1} x_t$.
By Corollary~\ref{cor:nomonodromy}, $x_{[\boldb_i \to c_i]}$ depends only
on $\boldb_i$ and $c_i$, not on the path chosen.

Since $\bolda_1$ is a multiset, the fibre $X(\bolda_1)$ is 
typically non-reduced, so there may be distinct points
$x, x' \in X(\bolda)$ with the same limit point
$x_{[\boldb_i \to c_i]} = x'_{[\boldb_i \to c_i]}$.
This defines an equivalence relation on the fibre $X(\bolda)$,
which we will study further in Section~\ref{sec:LRrule}.

\begin{theorem}
\label{thm:monodromylimit}
Let $\bolda$, $(\boldb, \boldb^c)$ be as in Example~\ref{ex:tableaucase},
and let $x, x' \in X(\bolda)$.
Suppose that $L \geq 2$, and $k \notin \{i-1, i, \dots, j\}$.
For any internal point $c_1$ for $\boldb$, we have
$x_{[\boldb \to c_1]} = x'_{[\boldb \to c_1]}$
if and only if 
$s_{k,L}(x)_{[\boldb \to c_1]} = 
s_{k,L}(x')_{[\boldb \to c_1]}$.
\end{theorem}

\begin{proof}
In order to see what happens to the point $x_t \in X(\bolda_t)$,
as $t$ approaches 1, we need to study the fibre $X(\bolda_t)$
when $|a_i|, \dots, |a_j|$ are close to each other.  Working over 
$\psK$, this corresponds to looking at $\psX(\bolda)$ where,
$\bolda = \{a_1, \dots, a_N\} \subset \psK^\times$ and
\begin{equation}
\label{eqn:assumptiononbolda}
\|a_1\| \leq  \dots \leq \|a_{i-1}\|
< \|a_i\| = \dots = \|a_j\|
< \|a_{j+1}\| \leq \dots \leq \|a_N\| \,.
\end{equation}
Let $T$ be a weakly increasing tableau of shape $\Rect$ with 
values in $\bolda$.  Let $\mu$ denote the shape of 
$T|_{\{a_1, \dots a_{i-1}\}}$, and let $\lambda$ be the shape of
$T|_{\{a_1, \dots a_j\}}$.  We will assume, moreover, that $T|_\mu$ 
and $T_{\lambda^c}$ are diagonally increasing.

We say $T$ corresponds
to a point $x \in \psX(\bolda)$ if~\eqref{eqn:valplucker}
holds.  
Since $T$ may not be a diagonally 
increasing tableau, Theorem~\ref{thm:leadtermequations} no longer 
provides us with the explicit system of equations needed to find
the leading terms of the points in $\psX(\bolda)$ corresponding to $T$.
However, it is still possible to write down such a system of equations
by following the same lines of argument.
The main difference between this case and the analysis in the proof
of Theorem~\ref{thm:leadtermequations} is that initial ideal of
the Pl\"ucker ideal $\initial_{\boldw(T)}(I)$ is not binomial;
it represents only a partial degeneration of
$X$ to the Gel'fand-Tsetlin toric variety.

We consider the initial forms of the 
equations~\eqref{eqn:satisfyplucker} and~\eqref{eqn:inthefibre}
to obtain a system of equations for the leading terms of the
Pl\"ucker coordinates of a point corresponding to $T$.  
This system will necessarily
be solvable, because the equations can be further degenerated to
the case of Theorem~\ref{thm:leadtermequations}, which is solvable.
The initial forms of~\eqref{eqn:satisfyplucker} are given by
$\initial_{\boldw(T)}(I)$.
We have
\begin{equation}
\label{eqn:partialGT}
p_\nu p_{\nu'} - p_{\nu \wedge \nu'} p_{\nu \vee \nu'} 
\in \initial_{\boldw(T)}(I) \qquad 
\text{if $\nu \leq \mu$ or $\nu \geq \lambda$\,.}
\end{equation}
The other
quadratic relations in $\initial_{\boldw(T)}(I)$ are more 
complicated;
however they only involve
partitions which are between $\mu$ and $\lambda$.  
Moreover the initial forms of~\eqref{eqn:inthefibre} only involve
partitions in this range.
From this, one can see that the system of equations one obtains 
for $\leadterm(p_\nu)$ for $\mu \leq \nu \leq \lambda$ depends only on 
$\{a_i, \dots, a_j\}$ and the shapes $\lambda$ and $\mu$.

Moreover given a solution to these equations, we can solve
for all remaining $\leadterm(p_\nu)$. 
For $\nu \leq \mu$, the equations determining $\leadterm(p_\nu)$,
up to a constant, are the same as those given by 
Theorem~\ref{thm:leadtermequations} applied to the tableau $T|_\mu$.
The constant is determined by the fact that we already have a value for
$\leadterm(p_\mu)$.
Similarly, for $\nu \geq \lambda$, the equations for $\leadterm(p_\nu)$
are given by Theorem~\ref{thm:leadtermequations} applied to $T|_{\lambda^c}$.
All other $\leadterm(p_\nu)$ are determined by~\eqref{eqn:partialGT}.

Now suppose that all inequalities in~\eqref{eqn:assumptiononbolda} are
strict, except for $\|a_k\| = \|a_{k+1}\|$.  Consider a loop 
$\bolda_t = \{(a_1)_t, \dots, (a_N)_t\}$, 
$t \in [0,1]$ in the space of multisubsets of $\PP^1(\psK)$,
with $\bolda_0 = \bolda_1 = \bolda$, and for all $t\in [0,1]$,
$(a_l)_t = a_l$ if $i \neq k, k+1$, 
$\|a_k\| = \|(a_k)_t\| = \|(a_{k+1})_t\| = \|a_{k+1}\|$.
Let $x_t, x'_t \in \psX(\bolda_t)$.  

Given a sufficiently small positive real number $\varepsilon$,
suppose $x_0$ and $x'_0$ are ``close together'', in that they have the
following properties: 
for all $\nu \in \Lambda$,
$\val(p_\nu(x_0)) = \val(p_\nu(x'_0))$, and
$$
1 - \varepsilon  < \left|\frac{\leadterm(p_\nu(x_0))}
{\leadterm(p_\nu(x'_0))}\right| < 1+\varepsilon
\,.
$$
Since the valuation of Pl\"ucker coordinates of $x_t$, $x'_t$ will be
independent of $t$, the points $x_t$ and $x'_t$ must correspond 
to the same weakly increasing tableau $T_t$.  
We claim, moreover, that 
\begin{equation}
\label{eqn:ltconstantratio}
\frac{\leadterm(p_\nu(x_0))}
{\leadterm(p_\nu(x'_0))} =
\frac{\leadterm(p_\nu(x_t))}
{\leadterm(p_\nu(x'_t))} 
\qquad
\text{for all $\nu \in \Lambda$, $t \in [0,1]$}.  
\end{equation}
This follows from the discussion above.
If $\mu \leq \nu \leq \lambda$, \eqref{eqn:ltconstantratio} is true because 
$\leadterm(p_\nu(x_t))$ and $\leadterm(p_\nu(x'_t))$ are both independent 
of $t$.  If $\nu \leq \mu$ (or
$\nu \geq \lambda$), \eqref{eqn:ltconstantratio} is true because 
$\leadterm(p_\nu(x_t))$ and $\leadterm(p_\nu(x'_t))$
must come from the same solution to~\eqref{eqn:leadtermequations}
for the tableau $T_t|_{\mu}$ (resp. $T_t|_{\lambda^c}$).  For all other 
$\nu \in \Lambda$, the claim follows from~\eqref{eqn:partialGT}.

In particular, we see that
$x_1$ and $x'_1$ are close together.
Taking our loop to be the loop
whose monodromy is given by $s_{k,L}$
(as defined in the proof of Theorem~\ref{thm:monodromyloops}),
the result follows.
\end{proof}

\section{Equivalence, dual equivalence, and the Littlewood-Richardson 
rule}
\label{sec:LRrule}

\subsection{Interpretations of equivalence and dual equivalence}

Throughout this section, we assume that
$\bolda  = \{a_1, \dots, a_N\} \subset \RP^1$, with
$|a_1|< \dots < |a_N|$.
We now show that in the situation in 
Example~\ref{ex:tableaucase},
the equivalence relations on $X(\bolda)$ defined in
Section~\ref{sec:limits} by 
$x_{[\boldb_i \to c_i]} = x'_{[\boldb_i \to c_i]}$ 
are combinatorially described by the
equivalence and dual equivalence relations on tableaux.

We will need the following lemma:
\begin{lemma}
\label{lem:realzeroinfinity}
Let $(\boldb_0, \boldb_\infty)$ be as in 
Example~\ref{ex:zeroinfinitycase}.  Let $T \in \SYT(\Rect; \bolda)$
and let $x_T \in X(\bolda)$ be the corresponding point.  Let
$\mu$ be the shape of $T|_{\boldb_0}$. 
\begin{enumerate}
\item[(i)] The point
$(x_T)_{[\boldb_0 \to 0]}$ is in $X_{\mu}(0)$ and corresponds to the 
tableau $T|_{\boldb_\infty}$.
\item[(ii)] The point
$(x_T)_{[\boldb_\infty \to \infty]}$ is in $X_{\mu^\vee}(\infty)$ and
corresponds to the tableau $T|_{\boldb_0}$.
\end{enumerate}
\end{lemma}

\begin{proof}
This follows from Theorem~\ref{thm:zeroinfinitylimit}.
\end{proof}

For $\phi \in \SL_2(\RR)$ and $T \in \SYT(\Rect; \bolda)$, define
$\phi(T) := \slide_{\phi(\bolda)}(T)$.  Here sliding is defined using
any path homotopic to a path of the form $\phi_t(\bolda)$,
$t \in [0,1]$, where $\phi_t \in \SL_2(\RR)$ is any path from
$\phi_0 = \smallidmatrix$
to $\phi_1 = \phi$.
From Theorems~\ref{thm:slidewelldefined} and~\ref{thm:geomslide},
we have that $\phi(x_T) = x_{\phi(T)}$; hence $\phi(T)$ does not depend 
on the choice of $\phi_t$. 

\begin{theorem}
\label{thm:equivalence}
Let $T, T' \in \SYT(\Rect; \bolda)$, and let $(\boldb, \boldb^c)$
be as in Example~\ref{ex:tableaucase}.
Choose any internal point $c_2$ for $\boldb^c$.
Let $x_T, x_{T'} \in X(\bolda)$ be the points
corresponding to $T$ and $T'$ respectively.  
Then 
$T|_\boldb \sim T'|_\boldb$ if and only if 
$(x_T)_{[\boldb^c \to c_2]} = (x_{T'})_{[\boldb^c \to c_2]}$.
\end{theorem}

\begin{proof}
Since the action of $\SL_2(\RR)$ on $\RP^1$ can take any three
points to any any other three points in the same orientation, 
there exists
$\phi \in \SL_2(\RR)$ be such that $\phi(c_2) = \infty$,
$|\phi(a)| < 1$ for $a \in \boldb$, and $|\phi(a)| > 1$ for 
$a \in \boldb^c$.

Consider $\phi(T)$ and $\phi(T')$.  We can compute these, via
a path $\bolda_t$ which first
rectifies $T|_\boldb, T'|_\boldb$.
By Theorem~\ref{thm:slidewelldefined}, it follows that
$\phi(T)|_{\phi(\boldb)} = \slide_{\phi(\boldb)}(\rectify(T|_\boldb))$ and
$\phi(T')|_{\phi(\boldb)} = \slide_{\phi(\boldb)}(\rectify(T'|_\boldb))$.
Thus we have that $T|_\boldb \sim T'|_\boldb$ if and only
if $\phi(T)|_{\phi(\boldb)} = \phi(T')|_{\phi(\boldb)}$.

By Proposition~\ref{prop:sl2equivariant}, 
$(x_T)_{[\boldb^c \to c_2]}  
= (x_{T'})_{[\boldb^c \to c_2]}$ if and only if 
$(x_{\phi(T)})_{[\phi(\boldb^c) \to \infty]}  
= (x_{\phi(T')})_{[\phi(\boldb^c) \to \infty]}$.   
By Lemma~\ref{lem:realzeroinfinity}(ii), 
this holds if
and only if $\phi(T)|_{\phi(\boldb)} = \phi(T')|_{\phi(\boldb)}$.
\end{proof}

\begin{remark}  \rm
\label{rmk:rectwelldefined}
Let $\boldc = \{a_1, \dots a_{i-1}\} \subset \bolda$ be the entries 
of $T$ to the left of $T|_\boldb$.  
As an addendum to the proof of
Theorem~\ref{thm:equivalence},
we give a quick proof of the fact that 
$\rectify(T|_\boldb) = \slide_{T|_\boldc}(T|_\boldb)$ does not
depend on $T|_\boldc$.  
\begin{proof}
Keeping the same notation, assume now that $c_2 = 0$.  
Consider $x_0 = (x_T)_{[\boldc \to 0]}$.
By Lemma~\ref{lem:realzeroinfinity}(i), 
$x_0$ corresponds to the tableau obtained by deleting the
entries in $\boldc$ from $T$.
Now, $\phi(x_0) = \phi(x_T)_{[\phi(\boldc) \to \infty]}$, so
by Lemma~\ref{lem:realzeroinfinity}(ii), 
$T_{\phi(x_0)}$ is the tableau obtained by deleting the
entries in $\phi(\boldc)$ from $\phi(T)$.  Since
$\rectify(T|_\boldb)$ can be determined from $\phi(T)|_{\phi(\boldb)}$,
it can also be determined from $x_0$, which
does not depend on $T|_\boldc$.
\end{proof}
\end{remark}

\begin{theorem}
\label{thm:dualequivalence}
Let $T, T' \in \SYT(\Rect; \bolda)$, and let $(\boldb, \boldb^c)$
be as in Example~\ref{ex:tableaucase}.
Choose any internal point $c_1$ for $\boldb$.
Let $x_T, x_{T'} \in X(\bolda)$ be the points
corresponding to $T$ and $T'$ respectively.  
\begin{enumerate}
\item[(i)] If $T|_{\boldb^c}  \neq T'|_{\boldb^c}$, then 
$(x_T)_{[\boldb \to c_1]} \neq (x_{T'})_{[\boldb \to c_1]}$.
\item[(ii)] If $T|_{\boldb^c} = T'|_{\boldb^c}$,
then $T|_\boldb \sim^* T'|_\boldb$ 
if and only if $(x_T)_{[\boldb \to c_1]}
= (x_{T'})_{[\boldb \to c_1]}$.
\item[(iii)] The point  $(x_T)_{[\boldb \to c_1]}$ is in  $X_\lambda(c_1)$, 
where $\lambda$ is the rectification shape of $T|_\boldb$.
\end{enumerate}
\end{theorem}

(Note that $T|_{\boldb^c}$, $T'|_{\boldb^c}$ will generally not be 
subtableaux of $T$ and $T'$.)

\begin{proof}
Let $\bolda_t  = \{(a_1)_t, \dots, (a_N)_t\} \Subset \RP^1$ be a 
path used to define 
$(x_T)_{[\boldb \to c_1]}$ and $(x_{T'})_{[\boldb \to c_1]}$.
Let $T_t = \slide_{\bolda_t}(T)$, $T'_t = \slide_{\bolda_t}(T')$.
Let $\boldb_t = \bolda_t \setminus \boldb^c$.
Assume that the path of each $(a_i)_t \in \boldb_t$ is monotonic.
Then for all of $t \in [0,1)$,
$T_t|_{\boldb_t} \sim^* T|_\boldb$, and 
$T'_t|_{\boldb_t} \sim^* T'|_\boldb$.  
We may therefore replace $T$ by $T_{1-\varepsilon}$; hence
we may assume that all elements of $\boldb$ are arbitrarily close to 
$c_1$.

Let $\phi \in \SL_2(\RR)$ be a transformation
such that $\phi(c_1) = 0$, and
consider $\phi(T)$ and $\phi(T')$.
Since the elements of $\phi(\boldb)$ are assumed to be
close to zero,
by Lemma~\ref{lem:realzeroinfinity}(i),
${\phi(T)}|_{\phi(\boldb^c)} = \phi(T')|_{\phi(\boldb^c)}$ 
if and only if 
$(x_{\phi(T)})_{[\phi(\boldb) \to 0]} = (x_{\phi(T')})_{[\phi(\boldb) \to 0]}$.
By Proposition~\ref{prop:sl2equivariant}, this holds if and only
if
$(x_T)_{[\boldb \to c_1]} = (x_{T'})_{[\boldb \to c_1]}$.
Thus, to prove (i) and (ii), we must therefore show that 
${\phi(T)}|_{\phi(\boldb^c)} = \phi(T')|_{\phi(\boldb^c)}$
if and only if
$T|_\boldb \sim^* T'|_\boldb$ and $T|_{\boldb^c} = T'|_{\boldb^c}$.

Let $\boldc = \{a_1, \dots, a_{i-1}\} \subset \bolda$, 
be the entries in the subtableau of $T$ to the left of $T|_\boldb$.
Let $\hat \boldc = \{a_{j+1}, \dots, a_N\} \subset \bolda$, 
be entries in the subtableau of $T$ to the right of $T|_\boldb$.
Note that ${\phi(T)}, \phi(T')$ can be computed by a path that  brings
the values in $\boldb$ past the values of $\boldc$ without changing
their relative order.  Thus
by definition of dual equivalence,
if $T|_\boldb \sim^* T'|_\boldb$ and $T|_{\boldb^c} = T'|_{\boldb^c}$
then ${\phi(T)}|_{\phi(\boldb^c)} = \phi(T')|_{\phi(\boldb^c)}$.

Conversely, suppose ${\phi(T)}|_{\phi(\boldb^c)} = \phi(T')|_{\phi(\boldb^c)}$.
We can recover $T|_{\boldb^c}$ and $T'|_{\boldb^c}$ by sliding
(the answer does not depend on 
${\phi(T)}|_{\phi(\boldb)}, \phi(T')|_{\phi(\boldb)}$);
hence we must have $T|_{\boldb^c} = T'|_{\boldb^c}$.
Moreover, from the argument of reverse direction, we see that 
$\slide_{T|_\boldb}(T|_{\boldc}) = \slide_{T|_\boldb}(T|_{\boldc})$.
By the same reasoning with $0$ replaced by $\infty$, we have
$\slide_{T|_\boldb}(T|_{\hat \boldc}) = \slide_{T|_\boldb}(T|_{\hat \boldc})$.

Let $\lambda/\mu$ be the shape of $T|_\boldb$ and $T'|_\boldb$.
To show that $T|_\boldb \sim^* T'|_\boldb$, we must show that
$\slide_{T|_\boldb}(V) = \slide_{T'|_\boldb}(V)$, for any tableau
$V$ in $\SYT(\mu; \boldc)$ or in $\SYT(\lambda^c; \hat \boldc)$.
Since we already know this for $V = T|_{\boldc}$ and $V = T_{\hat \boldc}$,
and since the operators $s_{k,L}$ act transitively on 
$\SYT(\mu; \boldc)$ and on $\SYT(\lambda^c; \hat \boldc)$,
the result now follows from Theorem~\ref{thm:monodromylimit}.

Finally, for (iii) we have already seen that
${\phi(T)}|_{\phi(\boldb^c)}$ has shape $\lambda^c$, where $\lambda$ is
the rectification shape of $T|_\boldb$. 
By Lemma~\ref{lem:realzeroinfinity}(i), we have that
$(x_{\phi(T)})_{[\phi(\boldb) \to 0]} \in X_\lambda(0)$,
and so
by Proposition~\ref{prop:sl2equivariant}, 
$(x_T)_{[\boldb \to c_1]} \in X_\lambda(c_1)$.
\end{proof}

\begin{remark} \rm
In the proof of Theorem~\ref{thm:dualequivalence}, we showed
that if $\slide_{T|_\boldb}(V) = \slide_{T'|_\boldb}(V)$ for some
$V \in \SYT(\mu; \boldc)$, then the same is true for every
$V \in \SYT(\mu; \boldc)$.  In fact our argument
shows that if $\slide_{T|_\boldb}(V) = \slide_{T'|_\boldb}(V)$ for 
any $V \in \SYT(\mu; \boldc)$, then
$(x_T)_{[\boldb \to c_1]} = (x_{T'})_{[\boldb \to c_1]}$, whence 
$T|_\boldb \sim^* T'|_\boldb$. This
combinatorial fact is a theorem of Haiman 
(see \cite[Theorem 2.10]{Hai}).

\end{remark}

\subsection{Combinatorial consequences}

A number of other combinatorial facts about 
equivalence and dual equivalence 
can be reproved using Theorems~\ref{thm:equivalence}
and~\ref{thm:dualequivalence}.

\begin{corollary}
\label{cor:dualequivsizes}
The size of a dual equivalence class with rectification shape
$\lambda$ is $|\ordSYT(\lambda)|$.
\end{corollary}

\begin{proof}
Let $(\boldb, \boldb^c)$, be a partition of $\bolda$, as in 
Example~\ref{ex:tableaucase}.
Let $c_1$ be an internal point for $\boldb$,
and let $\bolda_t$ be the path used to define $x_{[\boldb \to c_1]}$.
By Theorem~\ref{thm:dualequivalence}, a dual equivalence
class with rectification shape $\lambda$ corresponds to a point
in $X(\bolda_1)$ supported on $X_\lambda(c_1)$.
Since $\Wr$ is flat,
the size of the dual equivalence class is the multiplicity of the
point in $X(\bolda_1)$.
By Corollary~\ref{cor:svrootsmultiplicities}, the multiplicity of
such a point is $|\ordSYT(\lambda)|$.
\end{proof}

We can also prove a fact that was used in 
Section~\ref{sec:equivrelations} to give an alternate
formulation on the Littlewood-Richardson rule.

\begin{corollary}
\label{cor:classintersection}
There is a unique tableau in the intersection of any equivalence
class of tableaux with a dual equivalence class of the same rectification
shape.  
\end{corollary}

We need an additional lemma.

\begin{lemma}
\label{lem:geomclassintersection}
Let $(\boldb_1, \boldb_2)$ be a consecutive partition of
$\bolda \subset \RP^1$, and let
$c_1, c_2$ be internal points for $\boldb_1, \boldb_2$ respectively.
Let 
$x_1 \in X(\{c_1, \dots, c_1\} \cup \boldb_2)$,
and $x_2 = X(\boldb_1 \cup \{c_2, \dots, c_2\})$.
\begin{enumerate}
\item[(i)]
If $x_1 \in X_\lambda(c_1)$ and $x_2 \in X_{\lambda^\vee}(c_2)$ for
some $\lambda \in \Lambda$, then there exists a unique point
$x \in X(\bolda)$ such that $x_i = x_{[\boldb_i \to c_i]}$ for
$i =1,2$.
\item[(ii)]
If no such $\lambda$ exists then no such point $x$ exists.
\end{enumerate}
\end{lemma}

\begin{proof}
It suffices to prove this 
when $c_1 = 0$ and $c_2 = \infty$, and $|a| < |a'|$ for all 
$a \in \boldb_1$, $a' \in \boldb_2$.  
By Lemma~\ref{lem:realzeroinfinity}, 
if $x \in X(\bolda)$, then
$x_{[\boldb_2 \to c_2]}$ 
corresponds to the tableau $T_x|_{\boldb_1} = T_x|_\lambda$,
and $x_{[\boldb_1 \to c_1]}$ 
corresponds to tableau $T_x|_{\lambda^c}$.
Thus 
$x_{[\boldb_1 \to c_1]} \in X_\lambda(c_1)$ and
$x_{[\boldb_2 \to c_2]} \in X_{\lambda^\vee}(c_2)$, from which
(ii) follows.

To prove (i), suppose that 
$x_1 \in X_\lambda(c_1)$ and $x_2 \in X_{\lambda^\vee}(c_2)$.
Then $x_1$ corresponds to a tableau 
$T_{x_1} \in \SYT(\lambda^c; \boldb_2)$;
$x_2$ corresponds to 
a tableau $T_{x_2} \in \SYT(\lambda; \boldb_1)$.
There exists $T \in \SYT(\Rect; \bolda)$ 
such that $T|_{\boldb_i}= T_{x_i}$ for $i=1,2$.  Letting
$x = x_T$, by Lemma~\ref{lem:realzeroinfinity}, we have
$x_i = x_{[\boldb_i \to c_i]}$ for $i =1,2$, as required.

To prove uniqueness, we must show that
if $x, x' \in X(\bolda)$ and 
$x_{[\boldb_i \to c_i]} = x'_{[\boldb_i \to c_i]}$
for $i = 1,2$, then $x = x'$.
By Lemma~\ref{lem:realzeroinfinity} we have 
$T_x|_{\boldb_i} = T_{x'}|_{\boldb_i}$ for $i =1,2$; hence 
$T_x = T_{x'}$ which implies $x = x'$.
\end{proof}

\begin{proof}[Proof of Corollary~\ref{cor:classintersection}]
We will show that if $T \in \ordSYT(\lambda/\mu; \boldb)$ and 
$T' \in \ordSYT(\lambda'/\mu'; \boldb)$ both have the same rectification 
shape $\nu$, there is a unique tableau
$T'' \in \ordSYT(\lambda/\mu; \boldb)$ such that $T \sim^* T''$ and
$T' \sim T''$.

Choose a point  $x \in X$ such that  $\pi(x) \subset \RP^1$ and
$T_x|_\boldb = T$.  Let $x_1 = x_{[\boldb \to c_1]}$, where $c_1$ is 
an internal point for $\boldb$.
Choose $x' \in X$ such that $\pi(x') \subset \RP^1$ and
$T_{x'}|_\boldb = T'$.  Put $\bolda = \pi(x')$, 
$\boldb^c = \bolda \setminus \boldb$.
Let $x_2 = x_{[\boldb^c \to c_2]}$,
where $c_2$ is an internal point for $\boldb^c$.

Since $T$ and $T'$ have the same rectification shape $\nu$,
$x_1 \in X_\nu(c_1)$ and $x_2 \in X_{\nu^\vee}(c_2)$.
Thus, by Lemma~\ref{lem:geomclassintersection}(i)
there exists a unique point
$x'' \in X(\bolda)$ such that $x_i = x''_{[\boldb_i \to c_i]}$ for
$i =1,2$.  
By Theorems~\ref{thm:equivalence} and~\ref{thm:dualequivalence},
$T_{x''}|_\boldb \sim^* T$ and $T_{x''}|_\boldb \sim T'$.
\end{proof}

As a final note, recall from Remark~\ref{rmk:evacuation} that 
evacuation defines a $\ZZ$-action on standard Young tableaux shape $\Rect$.

\begin{corollary}
The evacuation action on $\ordSYT(\Rect)$ has order $N$.
\end{corollary}

\begin{proof}
Let $\xi = \frac{2\pi}{N}
\left(\begin{smallmatrix} 0 & 1 \\ -1 & 0 \end{smallmatrix}\right)$,
and
note that $e^{N\xi} = \smallidmatrix$.
Consider the loop $\bolda_t = \{(a_1)_t, \dots , (a_N)_t\}$, where
$(a_j)_t = \psi e^{(j-t) \xi}(0)$ and
$\psi \in \SL_2(\RR)$ is chosen so that
$0 < (a_1)_0 < \dots < (a_N)_0$.
Let $\phi = \psi e^\xi \psi^{-1} \in \SL_2(\RR)$.  
If $T \in \SYT(\Rect; \bolda_0)$,
then $\phi(T)$ is obtained by sliding $T$ using $\bolda_t$.  
But since $\bolda_t$ is a loop which cyclically rotates the elements
of $\bolda_0$, sliding using $\bolda_t$ gives the evacuation action 
on $T$.
The result follows, since $\phi^N = \psi (e^\xi)^N \psi^{-1} = \smallidmatrix$.
\end{proof}

\subsection{Proof of the Littlewood-Richardson rule}
\label{sec:lrrule}

Fix partitions $\lambda, \mu, \nu$, with $|\lambda| = |\mu| + |\nu|$.
Let $\bolda = \{a_1, \dots, a_N\}$, with  $0 < a_1 < \dots < a_N$.  
Let $\boldb_1 = \{a_1, \dots, a_{|\mu|}\}$,
$\boldb_2 = \{a_{|\mu|+1}, \dots , a_{|\lambda|}\}$ and 
$\boldb_3 = \{a_{|\lambda|+1}, \dots , a_N\}$.
Then $(\boldb_1, \boldb_2, \boldb_3)$ is a consecutive partition of 
$\bolda$.  Let $c_1, c_2, c_3$ be internal points of $\boldb_1,
\boldb_2, \boldb_3$ respectively.

We wish to count the number of points (with multiplicities)
in the intersection
$$Y = X_\mu(c_1) \cap X_\nu(c_2) \cap X_{\lambda^\vee}(c_3)\,.$$
This number is the Littlewood-Richardson number $c^\lambda_{\mu\nu}$.
First note that if a point $y \in Y$ has multiplicity $m$, then by
Corollary~\ref{cor:svrootsmultiplicities}, the point
$y$ has multiplicity 
$|\ordSYT(\mu)|\cdot|\ordSYT(\nu)|\cdot|\ordSYT(\lambda^\vee)|\cdot m$
in the intersection
$$\hat Y = X(c_1^{(|\mu|)}) \cap 
X(c_2^{(|\nu|)}) \cap 
X(c_3^{(|\lambda^\vee|)})\,.$$
Thus 
$c^\lambda_{\mu\nu} \cdot
|\ordSYT(\mu)|\cdot|\ordSYT(\nu)|\cdot|\ordSYT(\lambda^\vee)|$
is the number of points in $\hat Y$ 
that are supported on $Y$ (counted with multiplicities).

Since $\Wr$ is flat, the number of points in $\hat Y$ supported on $Y$
is the number of points
$x \in X(\bolda)$ such that 
$x_{[\boldb_1\to c_1][\boldb_2 \to c_2][\boldb_3 \to c_3]} \in Y$.
For a tableau $T \in \SYT(\Rect;\bolda)$,
let $x_T$ be the corresponding point in $X(\bolda)$.  
Then by Theorem~\ref{thm:dualequivalence},
$(x_T)_{[\boldb_1\to c_1][\boldb_2 \to c_2][\boldb_3 \to c_3]} \in Y$
if and only if the rectification shapes of 
$T|_{\boldb_1}$, $T|_{\boldb_2}$, $T|_{\boldb_3}$ are 
$\mu$, $\nu$ and $\lambda^\vee$ respectively.  
Let $S_{\mu\nu}^\lambda$ be the set of all  tableaux with this property.
The number of points in $\hat Y$ supported on $Y$ is 
therefore $|S_{\mu\nu}^\lambda|$.

Note that if $T \in S_{\mu\nu}^\lambda$, then
$T|_{\boldb_1}$ has shape $\mu$, and $T|_{\boldb_3}$ has shape 
$\lambda^c$.
Define an equivalence relation on $S_{\mu\nu}^\lambda$ by
putting $T \sim_2^* T'$ if $T|_{\boldb_2} \sim^* T'|_{\boldb_2}$.
But by Corollary~\ref{cor:dualequivsizes}, each equivalence class 
of $\sim_2^*$ has 
size 
$|[T]| = |\ordSYT(\mu)|\cdot|\ordSYT(\nu)|\cdot|\ordSYT(\lambda^\vee)|$.

Putting everything together, we have
\begin{multline*}
c^\lambda_{\mu\nu} \cdot
|\ordSYT(\mu)|\cdot|\ordSYT(\nu)|\cdot|\ordSYT(\lambda^\vee)|
=
|S_{\mu\nu}^\lambda| 
= \sum_{[T] \in (S_{\mu\nu}^\lambda/\sim_2^*)} |[T]|  \\
=
|\ordSYT(\mu)|\cdot|\ordSYT(\nu)|\cdot|\ordSYT(\lambda^\vee)|\cdot 
|S_{\mu\nu}^\lambda/\sim_2^*|\,.
\end{multline*}
Hence $c^\lambda_{\mu\nu} = |S_{\mu\nu}^\lambda/{\sim_2^*}|$, which is
precisely the statement of the Littlewood-Richardson rule,
as formulated in Theorem~\ref{thm:lrrule}.
\qed


\end{document}